\documentclass[11pt,a4,fleqn,english]{amsart}

\usepackage{amsmath}   
\usepackage{amsfonts,amssymb}

\usepackage[height=22cm , width = 16cm , top = 4cm , left = 3cm, a4paper]{geometry}
\usepackage[a4paper]{geometry}

\usepackage{graphicx} 
\usepackage{url}      

\usepackage{babel}
\usepackage{latexsym}

\input colordvi

\theoremstyle{plain}

\newtheorem{theo}{Theorem}[section]
\newtheorem{definition}[theo]{Definition}

\theoremstyle{remark}
\newtheorem{remark}[theo]{Remark}

\numberwithin{equation}{section}

\title[Convergence of CAT for Smoluchowski coagulation equation]{Convergence of the cell average technique for Smoluchowski coagulation equation}
\author{Ankik Kumar Giri\\}

\thanks{Johann Radon Institute for Computational and Applied Mathematics (RICAM), Austrian Academy of Sciences,
Altenbergerstrasse\ 69, A-4040 Linz, Austria} \email{ankik.giri@ricam.oeaw.ac.at, ankik.math@gmail.com}

\begin{document}
\baselineskip 16pt \markright{Text}

\thispagestyle{empty}

\begin{abstract}
We present the convergence analysis of the cell average technique, introduced in \cite{J_Kumar:2007}, to solve the
 nonlinear continuous Smoluchowski coagulation equation. It is shown that the technique is second order accurate on uniform grids 
and first order accurate on non-uniform smooth (geometric) grids. As an essential ingredient, the consistency of the technique is 
thoroughly discussed.
\end{abstract}

\maketitle

 \thispagestyle{empty}
 \section{Introduction}
 \label{sec:intro}
In this article we study some mathematical issues related to the convergence of the cell average technique (CAT) for solving the continuous Smoluchowski coagulation equation (SCE) which 
describes the dynamic evolution of particle growth. This model has many applications in biology, polymer science, 
astrophysics and oil industry etc. The nonlinear continuous SCE reads as 
\begin{eqnarray}\label{pbe}
\frac{\partial f(t,x)}{\partial t}  = \frac{1}{2}\int_{0}^{x} \beta(x-y,y)f(t,x-y)f(t,y)dy - \int_{0}^{\infty} \beta(x,y)f(t,x)f(t,y)dy,
\end{eqnarray}
with
\begin{eqnarray*}
f(x,0) = f^{\mbox{in}}(x)
\; \geq \; 0, \hspace{0.2in} x \in ]0,\infty[.
\end{eqnarray*}
Here the number density of particles of volume $x>0$ at time $t\geq 0$ is denoted by $f(x,t)$. The coagulation kernel 
$\beta(x,y)$ represents the rate at which particles of volume $x$ coalesce with particles of volume $y$. It will be 
assumed throughout the article that $\beta(x,y)=\beta(y,x)$ for all $x,y>0$, i.e.\ symmetric and $\beta(x,y)=0$ for 
either $x=0$ or $y=0$.
The integrals on the right-hand side of (\ref{pbe}) represent, respectively,
\begin{itemize}
 \item birth of particles of volume $x$ as a result of coagulation events of particles with volumes $y$ and $x-y$
$(0 \leq y \leq x)$
\item death of particles of volume $x$ due to the coagulation events with particles of volume $y$ 
$(0 \leq y < \infty)$.
\end{itemize}
There are several results available on the existence and uniqueness of solutions to (\ref{pbe}), see e.g. \cite{Dubovskii:1994, Dubovskii:1996, Escobedo:2003, Giri:2011, Giri:2011II, Giri:2012, Lamb:2004, Laurencot:2000, Laurencot:2002DC, 
Mclaughlin:1997I, Mclaughlin:1997, Stewart:1990}. To show all these results, one always needs certain growth conditions on the kernels. 
The SCE (\ref{pbe}) is analytically solvable only for some specific examples of kernels, see 
\cite{Dubovskii:1992, Fournier:2005, Fournier:2006}. 
Because of these restrictions, we are always interested to develop new numerical techniques with a detailed study 
of their mathematical analysis. Among all numerical methods for solving SCE (\ref{pbe}), the sectional methods are widely used 
because they are computationally attractive.
Recently, the cell average technique is introduced in \cite{J_Kumar:2007} which preserves all 
advantages of the existing sectional methods. Unlike the previous sectional methods, it gives very accurate prediction of selected higher moments and also provides quite satisfactory numerical results for 
number density. This gives us a strong motivation to analyze this technique mathematically. To apply a numerical method, first we need to consider the 
following truncated form of the problem (\ref{pbe}) by taking a finite computational domain $]0,R]$ where $0<R<\infty$.
\begin{eqnarray}\label{trunc pbe}
\frac{\partial n(t,x)}{\partial t}  =  \frac{1}{2}\int_{0}^{x} \beta(x-y,y)n(t,x-y)n(t,y)dy- \int_{0}^{R} \beta(x,y)n(t,x)n(t,y)dy,
\end{eqnarray}
with
\begin{eqnarray*}
n(x,0) = n^{\mbox{in}}(x)\geq 0, \ \ x\in \Omega:=]0,R],
\end{eqnarray*}
 where $n(t,x)$ represents the solution to the truncated equation (\ref{trunc pbe}). 
The existence and uniqueness of non-negative solutions for the truncated SCE (\ref{trunc pbe}) 
has been shown in \cite{Dubovskii:1996, Stewart:1990}. 
In \cite{Dubovskii:1996, Escobedo:2003, Giri:2011, Giri:2012, Laurencot:2000, Stewart:1990}, it is proven that the sequence 
of solutions to the truncated problems converge weakly to the solution of the original problem in a weighted $L^1$ space as $R \to \infty$
 for certain classes of kernels.

The purpose of this work is to demonstrate the convergence analysis of the cell average technique for solving SCE (\ref{trunc pbe}) on uniform and 
non-uniform smooth geometric grids. To the best of author's knowledge, this is the first attempt to show the convergence of CAT for solving nonlinear continuous 
SCE. The work presented here is motivated from \cite{J_Kumar:2008II} where the convergence of CAT is introduced for solving linear continuous breakage (fragmentation) 
problems.

The plan of this paper is as follows. The mathematical formulation of CAT is 
recalled in the next section. At the end of Section \ref{cat}, the main convergence theorem \ref{convergence theorem} is proved. To fulfill the 
requirements of main result, the consistency of the method and Lipschitz conditions are investigated in Section \ref{consistency} and \ref{convergence}, respectively. 
Finally, some conclusions are made in Section \ref{conc}.

\section{The cell average technique}\label{cat}
The cell average technique approximates the total number of particles in finite number of cells. 
As a first step, the continuous interval $\Omega:= ]0, R]$ is divided into a small number of cells defining size classes
$$\Lambda_i := ]x_{i-1/2}, x_{i+1/2}],\, i = 1, \ldots, I,$$
with
$$x_{1/2} = 0, \quad x_{I+1/2}= R.$$
The representative of each size class, usually the center point of each cell $x_i = (x_{i-1/2} + x_{i+1/2})/2$, is called pivot or
grid point.
We introduce $\Delta x_{\text{min}}$ and $\Delta x$ to satisfy
$$\Delta x_{\text{min}} \leq \Delta x_i = x_{i+1/2} - x_{i-1/2} \leq \Delta x.$$
For the purpose of later analysis we assume quasi uniformity of the grids, i.e.\
\begin{eqnarray}\label{grid cond}
 \frac{\Delta x}{\Delta x_{\text{min}}}\leq K,
\end{eqnarray}
where $K$ is a positive constant.
 The total number of particles in the $i$th cell is given as
\begin{eqnarray}\label{number}
N_i(t) = \int_{x_{i-1/2}}^{x_{i+1/2}} n(t, x) dx.
\end{eqnarray}
Integrating the continuous equation (\ref{trunc pbe}) over the $i$th cell we obtain
\begin{eqnarray*}
\frac{dN_i(t)}{dt} = B_i - D_i,\ \ \ \ i=1, \ldots, I.
\end{eqnarray*}
The total birth rate $B_i$ and the death rate $D_i$ are given as
\begin{eqnarray} \label{E:B_i}
B_i=\frac{1}{2}\int_{x_{i-1/2}}^{x_{i+1/2}}\int_{0}^{x}\beta(x-y, y)n(t,x-y)n(t,y)dy dx,
\end{eqnarray}
and
\begin{eqnarray}\label{E:D_i}
D_i = \int_{x_{i-1/2}}^{x_{i+1/2}}\int_{0}^{x_{I+1/2}}\beta(x, y)n(t,y)n(t,x)dy dx.
\end{eqnarray}
The above equations yield a semi-discrete system in $\mathbb{R}^I$
\begin{eqnarray}\label{E:IntegatedBreakage}
\frac{d\mathbf{N}}{dt} = \mathbf{B} - \mathbf{D}, \hspace{.1in} \mbox{with} \hspace{.1in} \mathbf{N}(0) = \mathbf{N}^{\mbox{in}},
\end{eqnarray}
where $\mathbf{N}, \mathbf{B}, \mathbf{D} \in \mathbb{R}^I$. 
The $i$th component of vectors $\mathbf{N}, \mathbf{B}$, and $\mathbf{D}$ are 
respectively defined in (\ref{number})-(\ref{E:D_i}). The vector $\mathbf{N}$ is formed by the vector of values of the step function obtained by $L^2$ projection of the exact solution $n$
into the space of step functions, which are constant on each cell. Note that this projection error can easily be shown of second order, see \cite{Giri:2010thesis}.
The total discrete birth and death rates of particles are evaluated by substituting the number density approximation
\begin{eqnarray*}
 n(t,x)\approx \sum_{i=1}^{I}N_i(t)\delta(x-x_i)
\end{eqnarray*}
into equations (\ref{E:B_i}) and (\ref{E:D_i}) as

\begin{eqnarray}\label{E:disc B_i}
 \hat{B}_i=\sum_{x_{i-1/2} \leq x_j+x_k < x_{i+1/2}}^{j\geq k}\bigg(1-\frac{1}{2}\delta_{j,k}\bigg)\beta(x_k,x_{j})N_j(t)N_k(t),
\end{eqnarray}
and
\begin{eqnarray}\label{E:disc D_i}
 \hat{D}_i=N_i(t)\sum_{j=1}^{I}\beta(x_i, x_j)N_j(t).
\end{eqnarray}
Here $\hat{B}_i$ and $\hat{D}_i$ denote the discrete birth and death rates, respectively, in the $i$th cell. The total volume flux $V_i$ into cell $i$ as a result of aggregation is given by
\begin{eqnarray}\label{E:V_i}
 V_i=\frac{1}{2}\int_{x_{i-1/2}}^{x_{i+1/2}}\int_{0}^{x}x\beta(x-y, y)n(t,x-y)n(t,y)dy dx.
\end{eqnarray}
Similarly to the discrete birth rate the discrete volume flux can be obtained as
\begin{eqnarray}\label{E:disc V_i}
 \hspace{1.1 cm}\hat{V}_i=\sum_{x_{i-1/2} \leq x_j+x_k < x_{i+1/2}}^{j\geq k}\bigg(1-\frac{1}{2}\delta_{j,k}\bigg)\beta(x_k,x_{j})N_j(t)N_k(t)(x_j+x_k).
\end{eqnarray}
Consequently, the average volume $\overline{v}_i\in
[x_{i-1/2},x_{i+1/2} ]$ of all new born particles in the $i$th cell
can be evaluated as
\begin{eqnarray}\label{def of v_i}
\overline{v}_i=\frac{\hat{V}_i}{\hat{B}_i},\ \ \hat{B}_i>0.
\end{eqnarray}
We do not need volume average $\overline{v}_i$ in case of
$\hat{B}_i=0$. However, for $\hat{B}_i=0$,
we can fix $\overline{v}_i=x_i$. Here we consider
that all new born particles in the $i$th cell are assigned temporarily to the average volume $\hat{v}_i$. If the average volume $\hat{v}_i$ is
same as the pivot point $x_i$ then the total birth $\hat{B}_i$  of the new born particles
can be
assigned to the pivot $x_i$ only. But this is rarely possible, and hence, the total birth $\hat{B}_i$ has to be assigned to the
neighboring pivots in such a way that the total number and mass remain
conserved during this reassignment. Finally, the resultant set of
ODEs takes the following form
\begin{eqnarray}\label{setofODES}
\frac{d\hat{N}_i}{dt} = \hat{B}^{CA}_i - \hat{D}^{CA}_i.
\end{eqnarray}
The above discretized system can also be written in the following vector form
\begin{eqnarray}\label{E:DiscretizedGenBreakage}
\frac{d\hat{\mathbf{N}}}{dt} = \hat{\mathbf{B}}(\hat{\mathbf{N}}) - \hat{\mathbf{D}}(\hat{\mathbf{N}})=:\hat{\mathbf {F}}(t,\hat{\mathbf{N}}), \hspace{.1in} \mbox{with} \hspace{.1in} \hat{\mathbf{N}}(0) = \mathbf{N}^{\mbox{in}},
\end{eqnarray}
where $ \hat{\mathbf{N}}, \hat{\mathbf{B}}, \hat{\mathbf{D}}\in \mathbb{R}^I$. The numerical approximation 
of total number of particles in $i$th cell, $N_i(t)$, is defined by $\hat{N}_i(t)$ which is the $i$th component of 
the vector $ \hat{\mathbf{N}}$. 
The discretized birth term, $\hat{B}^{CA}_i$, and death term, $\hat{D}^{CA}_i$, obtained from the cell average 
technique are defined below. These are the $i$th components of the vectors $\hat{\mathbf{B}}$ and $\hat{\mathbf{D}}$ respectively.
Let us consider the Heaviside function
\begin{eqnarray}\label{heavyside}
H(x):=\begin{cases}
1\ \, & \text{if}\ x>0,\\
\frac{1}{2}\ \, & \text{if}\ x=0,\\
\text{0}\ \, & \text{if}\ x<0.
\end{cases}
\end{eqnarray}
and
\begin{eqnarray}\label{values of lambda}
 \lambda_i^{\pm}(x)=\frac{x-x_{i\pm 1}}{x_i-x_{i\pm 1}}.
\end{eqnarray}
Then the birth and death terms are given as
\begin{eqnarray}\label{E:B_iCA}
\hat{B}^{CA}_i:&=&\hat{B}_{i-1}\lambda^{-}_i(\overline{v}_{i-1})H(\overline{v}_{i-1}-x_{i-1})
+\hat{B}_{i}\lambda^{+}_i(\overline{v}_{i})H(\overline{v}_{i}-x_{i})\\
&&\nonumber+\hat{B}_{i}\lambda^{-}_i(\overline{v}_{i})H(x_{i}-\overline{v}_{i})
+\hat{B}_{i+1}\lambda^{+}_i(\overline{v}_{i+1})H(x_{i+1}-\overline{v}_{i+1}),
\end{eqnarray}
and
\begin{eqnarray}\label{E:D_iCA}
\hat{D}^{CA}_i:= \hat{D}_i=N_i(t)\sum_{j=1}^{I}\beta(x_i, x_j)N_j(t).
\end{eqnarray}

The first and the fourth terms on the right hand side of equation
(\ref{E:B_iCA}) can be set to zero for $i=1$ and $i=I$, respectively. The detailed formulation can be found in \cite{J_Kumar:2007}. \\
By using (\ref{E:disc B_i}) and (\ref{E:disc D_i}) the cell average technique (\ref{setofODES}) can be written as
\begin{eqnarray}
 \frac{d\hat{N}_i(t)}{dt}&=&\lambda^{-}_i(\overline{v}_{i-1})H(\overline{v}_{i-1}-x_{i-1})\\
&&\hspace{1cm}\times \sum_{x_{i-3/2} \leq x_j+x_k < x_{i-1/2}}^{j\geq k}\bigg(1-\frac{1}{2}\delta_{j,k}\bigg)\beta(x_k,x_{j})\hat{N}_j(t)\hat{N}_k(t)\nonumber\\
&& + [\lambda^{+}_i(\overline{v}_{i})H(\overline{v}_{i}-x_{i})+\lambda^{-}_i(\overline{v}_{i})H(x_{i}-\overline{v}_{i}) ]\nonumber\\  
&&\hspace{1cm}\times \sum_{x_{i-1/2} \leq x_j+x_k < x_{i+1/2}}^{j\geq k}\bigg(1-\frac{1}{2}\delta_{j,k}\bigg)\beta(x_k,x_{j})\hat{N}_j(t)\hat{N}_k(t)\nonumber\\
&&+ \lambda^{+}_i(\overline{v}_{i+1})H(x_{i+1}-\overline{v}_{i+1})\nonumber\\ 
&& \hspace{1cm}\times \sum_{x_{i+1/2} \leq x_j+x_k < x_{i+3/2}}^{j\geq k}\bigg(1-\frac{1}{2}\delta_{j,k}\bigg)\beta(x_k,x_{j})\hat{N}_j(t)\hat{N}_k(t)\nonumber\\
&&\nonumber- \hat{N}_i(t)\sum_{j=1}^{I}\beta(x_i, x_j)\hat{N}_j(t).
\end{eqnarray}

It should be pointed out here that in this work we consider the following 
discrete norm 
\begin{eqnarray*}
 \|\mathbf{N}\|=\sum_{i=1}^{I}|N_i|.
\end{eqnarray*}
We consider $\mathcal{C}^2([a,b])$ as a space of twice continuously differentiable functions on $]a,b[$. Note that for the sake of simplicity in our analysis we assume that the
coagulation kernel satisfies
\begin{eqnarray}\label{cond on beta}
\beta \in \mathcal{C}^2(]0,R]\times
]0,R]).
\end{eqnarray}

Before moving to the main result, let us recall some definitions and a result taken from \cite{Hundsdorfer:2003}.
\begin{definition}
The \textbf{local discretization error} is defined by the residual left by substituting the exact solution
$\mathbf{N}(t)$ into equation (\ref{E:DiscretizedGenBreakage}) as
\begin{eqnarray}\label{E:SpatialTruncationError}
\mathbf{\sigma}(t) = \frac{d {\mathbf{N}}(t)}{d t} - \left(\hat{\mathbf{B}}\left({\mathbf{N}}(t)\right) -
\hat{\mathbf{D}}\left({\mathbf{N}}(t)\right)\right).
\end{eqnarray}
The scheme (\ref{E:DiscretizedGenBreakage}) is called consistent of order $p$ if, for $\Delta x \to 0$,
\begin{eqnarray*}
\| \mathbf{\sigma}(t) \| = {\mathcal O}(\Delta x^p), \quad \mbox{uniformly for all}\ \ t, \, 0 \leq t \leq T.
\end{eqnarray*}
\end{definition}

\begin{definition}
The \textbf{global discretization error} is defined by
\begin{eqnarray}\label{E:GlobalDiscretizationError}
 \mathbf{\epsilon}(t) = \mathbf{N}(t) -\hat{\mathbf{N}}(t).
\end{eqnarray}
 The scheme (\ref{E:DiscretizedGenBreakage}) is called convergent of order $p$ if, for $\Delta x \to 0$,
\begin{eqnarray*}
\| \mathbf{\epsilon}(t) \| = {\mathcal O}(\Delta x^p), \quad \mbox{uniformly for all } t,\, 0 \leq t \leq T.
\end{eqnarray*}
\end{definition}

It is important that the solution obtained by CAT remains non-negative for all times. This can be easily shown by using the next well known theorem.
In the following theorem we write $\hat{\mathbf{M}}\geq 0$ for a vector $\hat{\mathbf{M}}\in \mathbb{R}^I$ if all of its components are non-negative.
\begin{theo}\label{thm for positivity}
Suppose that $\hat{\mathbf {F}}(t,\hat{\mathbf{M}})$ defined in (\ref{E:DiscretizedGenBreakage}) is continuous and satisfies the Lipschitz condition as
\begin{eqnarray*}
 \|\hat{\mathbf{F}}(t,\hat{\mathbf{P}})-\hat{\mathbf{F}}(t,{\hat{\mathbf{M}}})\|\leq L \|\hat{\mathbf{P}}-\hat{\mathbf{M}}\|\ \ \mbox{ for all}\ \ \hat{\mathbf{P}}, \hat{\mathbf{M}}\in \mathbb{R}^I.
\end{eqnarray*}
Then the solution of the semi-discrete system (\ref{E:DiscretizedGenBreakage}) is non-negative if and only if for any vector $\hat{\mathbf{M}}\in \mathbb{R}^I$ with $\hat{\mathbf{M}}\geq 0$, 
the condition $\hat{M}_i =0$ implies $\hat{F}_i(t,\hat{\mathbf{M}}) \geq 0$ for any $i=1,\ldots,I$ and all $t\geq 0$.
\end{theo}
\begin{proof}
 The proof can be found in \cite[Chap. 1, Theorem 7.1]{Hundsdorfer:2003}.
\end{proof}

Now we shall state the main result which helps us to show the convergence of CAT. Note that the proof of Theorem \ref{convergence theorem} is motivated by a convergence result in \cite{Linz:1975}.
\begin{theo}\label{convergence theorem}
Let us assume that the Lipschitz conditions on $\hat{\mathbf{B}}({\mathbf{N}}(t))$ and $\hat{\mathbf{D}}({\mathbf{N}}(t))$ are satisfied for $0 \leq t \leq T$  and for all $\mathbf{N}$, $\hat{\mathbf{N}}\in \mathbb{R}^I$ where $\mathbf{N}$ and $\hat{\mathbf{N}}$ are the projected exact and numerical solutions defined in (\ref{E:IntegatedBreakage}) and (\ref{E:DiscretizedGenBreakage}) respectively. Then a consistent discretization method is  also convergent and the convergence is of the same order as the consistency.
\end{theo}
\begin{proof}
Using the equations (\ref{E:SpatialTruncationError}) and (\ref{E:GlobalDiscretizationError}) we have for $\mathbf{\epsilon}(t) = \mathbf{N}(t) -\hat{\mathbf{N}}(t)$
\begin{eqnarray*}
  \frac{d}{dt}\epsilon(t)=\sigma(t)+(\hat{\mathbf{B}}({\mathbf{N}})-\hat{\mathbf{B}}({\mathbf{\hat{N}}}))-(\hat{\mathbf{D}}({\mathbf{N}})-\hat{\mathbf{D}}({\mathbf{\hat{N}}})).
 \end{eqnarray*}

We then take the norm on both sides to get
\begin{eqnarray*}
  \frac{d}{dt}\|\epsilon(t)\|\leq\|\sigma(t)\|+\|(\hat{\mathbf{B}}({\mathbf{N}})-\hat{\mathbf{B}}({\mathbf{\hat{N}}}))\|+\|(\hat{\mathbf{D}}({\mathbf{N}})-\hat{\mathbf{D}}({\mathbf{\hat{N}}}))\|.
 \end{eqnarray*}

Integrating with respect to $t$ with $\epsilon(0)=0$ and using the Lipschitz conditions (\ref{lipschitz condition for B})-(\ref{lipschitz condition for D}) we obtain the estimates
\begin{eqnarray*}
 \|\epsilon(t)\|\leq \int_{0}^{t}\|\sigma(\tau)\|d \tau+ 2L \int_{0}^{t}\|\epsilon(\tau)\|d \tau.
\end{eqnarray*}

From this it follows by Gronwall's Lemma that
\begin{eqnarray}
 \|\epsilon(t)\|\leq \frac{e_h}{2L}[\exp{(2Lt)}-1],
\end{eqnarray}
where
\begin{eqnarray*}
 e_h=\max_{0\leq t \leq T} \|\sigma(t)\|.
\end{eqnarray*}

If the scheme is consistent then $\lim_{h\to 0}e_h=0$. This completes the proof.
\end{proof}
To fulfill the requirements of Theorem \ref{convergence theorem}, for the convergence of CAT we need to prove that the scheme is 
consistent as well as the birth $\hat{\mathbf{B}}({\mathbf{N}}(t))$ and death $\hat{\mathbf{D}}({\mathbf{N}}(t))$ terms satisfy the Lipschitz conditions.

\section{Consistency}\label{consistency}
In this section, the consistency of CAT for solving SCE (\ref{trunc pbe}) is discussed. The discretization 
error is evaluated in the birth and death terms, respectively. Then the local discretization error is calculated by 
considering uniform and non-uniform smooth geometric grids.

Let us begin with the integrated birth term of SCE (\ref{trunc pbe}) over $i$th cell
\begin{eqnarray*}
 B_i=\frac{1}{2}\int_{x_{i-1/2}}^{x_{i+1/2}}\int_{0}^{x}\beta(x-y,y)n(t,x-y)n(t,y)dy dx.
\end{eqnarray*}
By changing the order of integration we get
\begin{eqnarray*}
 B_i&=&\frac{1}{2}\sum_{j=1}^{i-1}\int_{x_{j-1/2}}^{x_{j+1/2}}\int_{x_{i-1/2}}^{x_{i+1/2}}\beta(x-y,y)n(t,x-y)n(t,y)dx dy \nonumber \\
&&+\frac{1}{2}\int_{x_{i-1/2}}^{x_{i+1/2}}\int_{y}^{x_{i+1/2}}\beta(x-y,y)n(t,x-y)n(t,y)dx dy.
\end{eqnarray*}
Now we apply the midpoint rule to the outer integrals in both terms on the right-hand 
side and use the relationship $N_i=n(t,x_i)\Delta x_i + {\mathcal O}(\Delta x^3)$ for the midpoint rule to obtain
\begin{eqnarray}\label{B_i in two integrals}
 B_i&=&\frac{1}{2}\sum_{j=1}^{i-1}N_j(t)\int_{x_{i-1/2}}^{x_{i+1/2}}\beta(x-x_j,x_j)n(t,x-x_j)dx \\
&&+\frac{1}{2}N_i(t)\int_{x_{i}}^{x_{i+1/2}}\beta(x-x_i,x_i)n(t,x-x_i)dx +{\mathcal O}(\Delta x^3),\nonumber\\
\nonumber&=:&\tilde{B}_i +{\mathcal O}(\Delta x^3).
\end{eqnarray}
Let us denote the integral terms in $\tilde{B}_i$ by $I_1$ and $I_2$, respectively, and evaluate them separately.

{\bf{Integral term $I_1$:}} We consider the first integral term on the right-hand side in 
(\ref{B_i in two integrals}) and use the substitution $x-x_j=x'$ to get
\begin{eqnarray}\label{I_1}
 I_1=\frac{1}{2}\sum_{j=1}^{i-1}N_j(t)\int_{x_{i-1/2}-x_j}^{x_{i+1/2}-x_j}\beta(x',x_j)n(t,x')dx'.
\end{eqnarray}
We now define $l_{i,j}$ and $\gamma_{i,j}$ to be those indices such that the following hold
\begin{eqnarray}\label{defi of index}
 x_{i-1/2}-x_j \in \Lambda_{l_{i,j}}\ \ \text{and} \ \ \gamma_{i,j}:=\text{sgn} [(x_{i-1/2}-x_j)- x_{l_{i,j}}],
\end{eqnarray}
where
\begin{eqnarray*}
\text{sgn}(x):=\begin{cases}
1\ \, & \text{if}\ x>0,\\
0 \ \, & \text{if}\ x=0,\\
\text{-1}\ \, & \text{if}\ x<0.
\end{cases}
\end{eqnarray*}
By the definition of the indices $l_{i,j}$ and $\gamma_{i,j}$ in (\ref{defi of index}), the equation (\ref{I_1}) can be rewritten as
 \begin{eqnarray}\label{I_1 in l,gamma}
 I_1&=&\frac{1}{2}\sum_{j=1}^{i-1} N_j(t) \int_{x_{i-1/2}-x_j}^{x_{l_{i,j}+\frac{1}{2}\gamma_{i,j}}}\beta(x',x_j)n(t,x')dx'\\
 &&+\frac{1}{2}\sum_{j=1}^{i-1}N_j(t)\sum_{k={l_{i,j}+\frac{1}{2}(\gamma_{i,j}+1)}}^{l_{i+1,j}+\frac{1}{2}(\gamma_{i+1,j}-1)}\int_{x_{k-1/2}}^{x_{k+1/2}}\beta(x',x_j)n(t,x')dx'\nonumber\\
 &&\nonumber+\frac{1}{2}\sum_{j=1}^{i-1}N_j(t)\int_{x_{l_{i+1,j}+\frac{1}{2}\gamma_{i+1,j}}}^{x_{i+1/2}-x_j}\beta(x',x_j)n(t,x')dx'.
\end{eqnarray}
Let $p$ be the total number of terms in the following sum
\begin{eqnarray*}
\sum_{k={l_{i,j}+\frac{1}{2}(\gamma_{i,j}+1)}}^{l_{i+
1,j}+\frac{1}{2}(\gamma_{i+1,j}-1)}\int_{x_{k-1/2}}^{x_{k+1/2}}\beta(x',x_j)n(t,x')dx'.
\end{eqnarray*}
In particular, let $p:= \#  \{ n:
l_{i,j}+\frac{1}{2}(\gamma_{i,j}+1) \leq n \leq
l_{i+1,j}+\frac{1}{2}(\gamma_{i+1,j}-1)\}$
 and set
$$k_1:=l_{i,j}+\frac{1}{2}(\gamma_{i,j}+1),\ \ k_2:=k_1+1,\ldots,\ \ k_{p-1}:=k_1+(p-2).$$
Next, we shall show that $p$ is finite and can be estimated by a constant which is independent of the grid size.
By using the definition of the indices $l_{i,j}$ and $\gamma_{i,j}$ in (\ref{defi of index}), we can estimate
$$(p-2)\Delta x_{\mbox{min}}\leq \Delta x_{k_2}+ \Delta x_{k_3}+ \ldots + \Delta x_{k_{p-1}} \leq \frac{1}{2}(\Delta x_{i}+\Delta x_{i+1})\leq \Delta x$$
which implies using the assumption of quasi uniformity (\ref{grid cond}) that
$$(p-2)\leq \frac{\Delta x}{\Delta x_{\mbox{min}}}\leq K \Rightarrow  p \leq K+2.$$
This means the above sum has uniformly bounded finite number of
terms. So we can apply the midpoint rule to the integral in second term on
the right hand side and use $N_k(t)=n(t,x_k)\Delta x_k + {\mathcal O}(\Delta x^3)$ to get
\begin{eqnarray}\label{final I_1}
 I_1&=&\frac{1}{2}\sum_{j=1}^{i-1} N_j(t) \int_{x_{i-1/2}-x_j}^{x_{l_{i,j}+\frac{1}{2}\gamma_{i,j}}}\beta(x',x_j)n(t,x')dx'\\
 &&+\frac{1}{2}\sum_{j=1}^{i-1}N_j(t)\sum_{x_{i-1/2}\leq (x_j+x_k)<x_{i+1/2}}\beta(x_k,x_j)N_k(t) \nonumber\\
 &&\nonumber+\frac{1}{2}\sum_{j=1}^{i-1}N_j(t)\int_{x_{l_{i+1,j}+\frac{1}{2}\gamma_{i+1,j}}}^{x_{i+1/2}-x_j}\beta(x',x_j)n(t,x')dx' +{\mathcal O}(\Delta x^3).
\end{eqnarray}

{\bf{Integral term $I_2$:}} Let us consider the second integral term in 
(\ref{B_i in two integrals}) and use the substitution $x-x_i=x'$ to estimate
\begin{eqnarray*}
I_2=\frac{1}{2}N_i(t)\int_{0}^{x_{i+1/2}-x_i}\beta(x',x_i)n(t,x')dx.
\end{eqnarray*}
Again by the definition of the indices $l_{i,j}$ and $\gamma_{i,j}$ in (\ref{defi of index}) we split the above integral as
\begin{eqnarray*}
 I_2&=&\frac{1}{2}N_i(t)\sum_{k=1}^{l_{i+1,i}+\frac{1}{2}(\gamma_{i+1,i}-1)}\int_{x_{k-1/2}}^{x_{k+1/2}}\beta(x',x_i)n(t,x')dx'\\
 &&+\frac{1}{2}N_i(t)\int_{x_{l_{i+1,i}+\frac{1}{2}\gamma_{i+1,i}}}^{x_{i+1/2}-x_i}\beta(x',x_i)n(t,x')dx'.
\end{eqnarray*}
By applying the midpoint rule in the first term and using the definition of the indices $l_{i,j}$ and $\gamma_{i,j}$, we get
\begin{eqnarray}\label{final I_2}
 I_2&=&\frac{1}{2}N_i(t)\sum_{x_i+x_k < x_{i+1/2}}\beta(x_k,x_i)N_k(t)\\
 &&\nonumber+\frac{1}{2}N_i(t)\int_{x_{l_{i+1,i}+\frac{1}{2}\gamma_{i+1,i}}}^{x_{i+1/2}-x_i}\beta(x',x_i)n(t,x')dx' +{\mathcal O}(\Delta x^3).
\end{eqnarray}
By substituting (\ref{final I_1}), (\ref{final I_2}) into (\ref{B_i in two integrals}) and using (\ref{E:disc B_i}), we estimate
\begin{eqnarray*}
B_i&=&\frac{1}{2}\sum_{j=1}^{i-1}N_j(t)\sum_{x_{i-1/2}\leq (x_j+x_k)<x_{i+1/2}}\beta(x_k,x_j)N_k(t) \nonumber\\
&&+\frac{1}{2}N_i(t)\sum_{x_i+x_k < x_{i+1/2}}\beta(x_k,x_i)N_k(t)\nonumber\\
&&+\frac{1}{2}\sum_{j=1}^{i-1} N_j(t) \int_{x_{i-1/2}-x_j}^{x_{l_{i,j}+\frac{1}{2}\gamma_{i,j}}}\beta(x',x_j)n(t,x')dx'\nonumber\\
&&+\frac{1}{2}\sum_{j=1}^{i-1}N_j(t)\int_{x_{l_{i+1,j}+\frac{1}{2}\gamma_{i+1,j}}}^{x_{i+1/2}-x_j}\beta(x',x_j)n(t,x')dx'\\ 
&&+\frac{1}{2}N_i(t)\int_{x_{l_{i+1,i}+\frac{1}{2}\gamma_{i+1,i}}}^{x_{i+1/2}-x_i}\beta(x',x_i)n(t,x')dx' +{\mathcal O}(\Delta x^3).
\end{eqnarray*}
The first two terms on the right hand side can be combined as
\begin{eqnarray*}
B_i&=&\sum_{x_{i-1/2} \leq x_j+x_k < x_{i+1/2}}^{j\geq k}\bigg(1-\frac{1}{2}\delta_{j,k}\bigg)\beta(x_k,x_{j})N_j(t)N_k(t)\\
&&+\frac{1}{2}\sum_{j=1}^{i-1} N_j(t) \int_{x_{i-1/2}-x_j}^{x_{l_{i,j}+\frac{1}{2}\gamma_{i,j}}}\beta(x',x_j)n(t,x')dx'\nonumber\\
&&+\frac{1}{2}\sum_{j=1}^{i}N_j(t)\int_{x_{l_{i+1,j}+\frac{1}{2}\gamma_{i+1,j}}}^{x_{i+1/2}-x_j}\beta(x',x_j)n(t,x')dx' +{\mathcal O}(\Delta x^3).
\end{eqnarray*}
Finally, by using (\ref{E:disc B_i}) the above expression for $B_i$ can be written as
\begin{eqnarray}\label{final B_i}
B_i&=&\hat{B}_i+\frac{1}{2}\sum_{j=1}^{i-1} N_j(t) \int_{x_{i-1/2}-x_j}^{x_{l_{i,j}+\frac{1}{2}\gamma_{i,j}}}\beta(x',x_j)n(t,x')dx'\\
&&\nonumber+\frac{1}{2}\sum_{j=1}^{i}N_j(t)\int_{x_{l_{i+1,j}+\frac{1}{2}\gamma_{i+1,j}}}^{x_{i+1/2}-x_j}\beta(x',x_j)n(t,x')dx' +{\mathcal O}(\Delta x^3).
\end{eqnarray}
Let us denote the sum of the remaining two integrals on the right hand side in (\ref{final B_i}) by the error $E_1$ which will be discussed later. 

Now we concentrate to evaluate the integrated term $V_i-x_i B_i$ 
by using (\ref{E:B_i}) and (\ref{E:V_i}) as follows
\begin{eqnarray*}
\hspace{2cm} V_i-x_i B_i=\frac{1}{2}\int_{x_{i-1/2}}^{x_{i+1/2}}\int_{0}^{x}(x-x_i)\beta(x-y, y)n(t,x-y)n(t,y)dy dx.
\end{eqnarray*}
By changing the order of integration we get
\begin{eqnarray*}
V_i-x_i B_i&=&\frac{1}{2}\sum_{j=1}^{i-1}\int_{x_{j-1/2}}^{x_{j+1/2}}\int_{x_{i-1/2}}^{x_{i+1/2}}(x-x_i)\beta(x-y,y)n(t,x-y)n(t,y)dx dy \nonumber \\
&&+\frac{1}{2}\int_{x_{i-1/2}}^{x_{i+1/2}}\int_{y}^{x_{i+1/2}}(x-x_i)\beta(x-y,y)n(t,x-y)n(t,y)dx dy.
\end{eqnarray*}
Now applying the midpoint rule to the outer integrals in both the terms on the right hand side and using the 
relationship $N_i=n(t,x_i)\Delta x_i + {\mathcal O}(\Delta x^3)$ with $\beta(0,\cdot)=0$, we obtain 
\begin{eqnarray}\label{V_i in two integrals}
 V_i-x_iB_ i&=&\frac{1}{2}\sum_{j=1}^{i-1}N_j(t)\int_{x_{i-1/2}}^{x_{i+1/2}}(x-x_i)\beta(x-x_j,x_j)n(t,x-x_j)dx \\
&&+\frac{1}{2}N_i(t)\int_{x_{i}}^{x_{i+1/2}}(x-x_i)\beta(x-x_i,x_i)n(t,x-x_i)dx +{\mathcal O}(\Delta x^5),\nonumber\\
\nonumber&=:&\tilde{V}_i-x_i\tilde{B}_i +{\mathcal O}(\Delta x^5).
\end{eqnarray}
We denote the integral terms involving in $\tilde{V}_i-x_i\tilde{B}_i$ by $P_1$ and $P_2$, respectively, and 
calculate them separately.

{\bf{Integral term $P_1$:}}
Let us consider the first integral term in (\ref{V_i in two integrals}) and insert $x-x_j=x'$ to estimate 
\begin{eqnarray}\label{P_1}
 P_1=\frac{1}{2}\sum_{j=1}^{i-1}N_j(t)\int_{x_{i-1/2}-x_j}^{x_{i+1/2}-x_j}(x'-x_i+x_j)\beta(x',x_j)n(t,x')dx'.
\end{eqnarray}
By the definition of the indices $l_{i,j}$ and $\gamma_{i,j}$ in (\ref{defi of index}), (\ref{P_1}) can be rewritten 
as 
 \begin{eqnarray*}
 P_1&=&\frac{1}{2}\sum_{j=1}^{i-1} N_j(t) \int_{x_{i-1/2}-x_j}^{x_{l_{i,j}+\frac{1}{2}\gamma_{i,j}}}(x'-x_i+x_j)\beta(x',x_j)n(t,x')dx'\nonumber\\
& &+\frac{1}{2}\sum_{j=1}^{i-1}N_j(t)\sum_{k={l_{i,j}+\frac{1}{2}(\gamma_{i,j}+1)}}^{l_{i+1,j}+\frac{1}{2}(\gamma_{i+1,j}-1)}\int_{x_{k-1/2}}^{x_{k+1/2}}(x'-x_i+x_j)\beta(x',x_j)n(t,x')dx'\nonumber\\
& &+\frac{1}{2}\sum_{j=1}^{i-1}N_j(t)\int_{x_{l_{i+1,j}+\frac{1}{2}\gamma_{i+1,j}}}^{x_{i+1/2}-x_j}(x'-x_i+x_j)\beta(x',x_j)n(t,x')dx'.
\end{eqnarray*}
Since the number of terms in the inner summation of second term on the right hand side is finite as before, therefore 
we can use the midpoint rule to the integral in second term on
the right hand side and use $N_k(t)=n(t,x_k)\Delta x_k + {\mathcal O}(\Delta x^3)$ to obtain
\begin{eqnarray}\label{final P_1}
 P_1&=&\frac{1}{2}\sum_{j=1}^{i-1} N_j(t) \int_{x_{i-1/2}-x_j}^{x_{l_{i,j}+\frac{1}{2}\gamma_{i,j}}}(x'-x_i+x_j)\beta(x',x_j)n(t,x')dx'\\
& &+\frac{1}{2}\sum_{j=1}^{i-1}N_j(t)\sum_{x_{i-1/2}\leq (x_j+x_k)<x_{i+1/2}}(x_k-x_i+x_j)\beta(x_k,x_j)N_k(t) \nonumber\\
& &+\frac{1}{2}\sum_{j=1}^{i-1}N_j(t)\int_{x_{l_{i+1,j}+\frac{1}{2}\gamma_{i+1,j}}}^{x_{i+1/2}-x_j}(x'-x_i+x_j)\beta(x',x_j)n(t,x')dx'\nonumber\\
& &+\frac{1}{24}\sum_{j=1}^{i-1}N_j(t)\sum_{k={l_{i,j}+\frac{1}{2}(\gamma_{i,j}+1)}}^{l_{i+1,j}+\frac{1}{2}(\gamma_{i+1,j}-1)}\Delta {x_k}^3 
\frac{\partial}{\partial x'}{\{\beta(x_k,x_j)n(t,x_k)\}}\nonumber\\
&&\nonumber+{\mathcal O}(\Delta x^4).
\end{eqnarray}

{\bf{Integral term $P_2$:}} 
Let us consider the second integral term in 
(\ref{V_i in two integrals}) and use the substitution $x-x_i=x'$ to estimate
\begin{eqnarray*}
P_2=\frac{1}{2}N_i(t)\int_{0}^{x_{i+1/2}-x_i}x'\beta(x',x_i)n(t,x')dx.
\end{eqnarray*}
By the definition of the indices $l_{i,j}$ and $\gamma_{i,j}$ in (\ref{defi of index}) we split the above integral as
\begin{eqnarray*}
 P_2&=&\frac{1}{2}N_i(t)\sum_{k=1}^{l_{i+1,i}+\frac{1}{2}(\gamma_{i+1,i}-1)}\int_{x_{k-1/2}}^{x_{k+1/2}}x'\beta(x',x_i)n(t,x')dx'\\
& &+\frac{1}{2}N_i(t)\int_{x_{l_{i+1,i}+\frac{1}{2}\gamma_{i+1,i}}}^{x_{i+1/2}-x_i}x'\beta(x',x_i)n(t,x')dx'.
\end{eqnarray*}
We apply the midpoint rule in the first term on the right hand side to obtain
\begin{eqnarray}\label{final P_2}
 P_2&=&\frac{1}{2}N_i(t)\sum_{x_i+x_k < x_{i+1/2}}(x_i+x_k-x_i)\beta(x_k,x_i)N_k(t)\\
& &\nonumber+\frac{1}{2}N_i(t)\int_{x_{l_{i+1,i}+\frac{1}{2}\gamma_{i+1,i}}}^{x_{i+1/2}-x_i}x'\beta(x',x_i)n(t,x')dx' +{\mathcal O}(\Delta x^4).
\end{eqnarray}
By substituting (\ref{final P_1}), and (\ref{final P_2}) into (\ref{V_i in two integrals}), we estimate
\begin{eqnarray*}
V_i-x_iB_i&=&\frac{1}{2}\sum_{j=1}^{i-1}N_j(t)\sum_{x_{i-1/2}\leq (x_j+x_k)<x_{i+1/2}}(x_j+x_k-x_i)\beta(x_k,x_j)N_k(t) \nonumber\\
&&+\frac{1}{2}N_i(t)\sum_{x_i+x_k < x_{i+1/2}}(x_i+x_k-x_i)\beta(x_k,x_i)N_k(t)\nonumber\\
&&+\frac{1}{2}\sum_{j=1}^{i-1} N_j(t) \int_{x_{i-1/2}-x_j}^{x_{l_{i,j}+\frac{1}{2}\gamma_{i,j}}}(x'-x_i+x_j)\beta(x',x_j)n(t,x')dx'\nonumber\\
&&+\frac{1}{2}\sum_{j=1}^{i}N_j(t)\int_{x_{l_{i+1,j}+\frac{1}{2}\gamma_{i+1,j}}}^{x_{i+1/2}-x_j}(x'-x_i+x_j)\beta(x',x_j)n(t,x')dx'\\ 
&&+\frac{1}{24}\sum_{j=1}^{i-1}N_j(t)\sum_{k={l_{i,j}+\frac{1}{2}(\gamma_{i,j}+1)}}^{l_{i+1,j}+\frac{1}{2}(\gamma_{i+1,j}-1)}\Delta {x_k}^3 
\frac{\partial}{\partial x'}{\{\beta(x_k,x_j)n(t,x_k)\}} +{\mathcal O}(\Delta x^4).
\end{eqnarray*}
The first two terms on the right-hand side can be combined and the above equation can be written as
\begin{eqnarray*}
V_i-x_iB_i&=&\sum_{x_{i-1/2} \leq x_j+x_k < x_{i+1/2}}^{j\geq k}\bigg(1-\frac{1}{2}\delta_{j,k}\bigg)(x_j+x_k-x_i)\beta(x_k,x_{j})N_j(t)N_k(t)\\
&&+\frac{1}{2}\sum_{j=1}^{i-1} N_j(t) \int_{x_{i-1/2}-x_j}^{x_{l_{i,j}+\frac{1}{2}\gamma_{i,j}}}(x'-x_i+x_j)\beta(x',x_j)n(t,x')dx'\nonumber\\
&&+\frac{1}{2}\sum_{j=1}^{i}N_j(t)\int_{x_{l_{i+1,j}+\frac{1}{2}\gamma_{i+1,j}}}^{x_{i+1/2}-x_j}(x'-x_i+x_j)\beta(x',x_j)n(t,x')dx'\\ 
&&+\frac{1}{24}\sum_{j=1}^{i-1}N_j(t)\sum_{k={l_{i,j}+\frac{1}{2}(\gamma_{i,j}+1)}}^{l_{i+1,j}+\frac{1}{2}(\gamma_{i+1,j}-1)}\Delta {x_k}^3 
\frac{\partial}{\partial x'}{\{\beta(x_k,x_j)n(t,x_k)\}} +{\mathcal O}(\Delta x^4).
\end{eqnarray*}
Finally, by using (\ref{E:disc V_i}) and (\ref{E:disc B_i}), the above expression can be rewritten as
\begin{eqnarray}\label{eq V_i and Disc_V_i}
V_i-x_iB_i&= &\hat{V}_i-x_i \hat{B}_i\\
&&+\frac{1}{2}\sum_{j=1}^{i-1} N_j(t) \int_{x_{i-1/2}-x_j}^{x_{l_{i,j}+\frac{1}{2}\gamma_{i,j}}}(x'-x_i+x_j)\beta(x',x_j)n(t,x')dx'\nonumber\\
&&\frac{1}{2}\sum_{j=1}^{i}N_j(t)\int_{x_{l_{i+1,j}+\frac{1}{2}\gamma_{i+1,j}}}^{x_{i+1/2}-x_j}(x'-x_i+x_j)\beta(x',x_j)n(t,x')dx'\nonumber\\ 
&&+\frac{1}{24}\sum_{j=1}^{i-1}N_j(t)\sum_{k={l_{i,j}+\frac{1}{2}(\gamma_{i,j}+1)}}^{l_{i+1,j}+\frac{1}{2}(\gamma_{i+1,j}-1)}\Delta {x_k}^3 
\frac{\partial}{\partial x'}{\{\beta(x_k,x_j)n(t,x_k)\}}\nonumber\\
&&\nonumber +{\mathcal O}(\Delta x^4).
\end{eqnarray}
Now we evaluate the each term in (\ref{E:B_iCA}) separately. We begin with the first term 
without Heaviside function $H(x)$ and insert the value of $\lambda$ from (\ref{values of lambda}) to get 
\begin{eqnarray*}
 \lambda^{-}_{i}(\overline{v}_{i-1})\hat{B}_{i-1}=\frac{\overline{v}_{i-1}-x_{i-1}}{x_i-x_{i-1}}\hat{B}_{i-1}
=\frac{2}{\Delta x_i + \Delta x_{i-1}}[\hat{V}_{i-1}-x_{i-1}\hat{B}_{i-1}].
\end{eqnarray*}
Using the equation (\ref{V_i in two integrals}) and (\ref{eq V_i and Disc_V_i}), we obtain
\begin{eqnarray}\label{1st eq lambda}
\hspace{.2in} \lambda^{-}_{i}(\overline{v}_{i-1})\hat{B}_{i-1}&=&\frac{2}{\Delta x_i + \Delta x_{i-1}}\bigg[\tilde{V}_{i-1}-x_{i-1}\tilde{B}_{i-1}\\
&&-\frac{1}{2}\sum_{j=1}^{i-2} N_j(t) \int_{x_{i-3/2}-x_j}^{x_{l_{i-1,j}+\frac{1}{2}\gamma_{i-1,j}}}(x'-x_{i-1}+x_j)\beta(x',x_j)n(t,x')dx'\nonumber\\
&&-\frac{1}{2}\sum_{j=1}^{i-1}N_j(t)\int_{x_{l_{i,j}+\frac{1}{2}\gamma_{i,j}}}^{x_{i-1/2}-x_j}(x'-x_{i-1}+x_j)\beta(x',x_j)n(t,x')dx'\nonumber\\ 
&&-\frac{1}{24}\sum_{j=1}^{i-2}N_j(t)\sum_{k={l_{i-1,j}+\frac{1}{2}(\gamma_{i-1,j}+1)}}^{l_{i,j}+\frac{1}{2}(\gamma_{i,j}-1)}\Delta {x_k}^3 
\frac{\partial}{\partial x'}{\{\beta(x_k,x_j)n(t,x_k)\}}\nonumber\\
&&\nonumber \hspace{.4in}+{\mathcal O}(\Delta x^4)\bigg].
\end{eqnarray}
To solve equation (\ref{1st eq lambda}), we estimate $\tilde{V}_{i-1}-x_{i-1}\tilde{B}_{i-1}$ by using 
(\ref{V_i in two integrals}) as follows
\begin{eqnarray*}
 \tilde{V}_{i-1}-x_{i-1}\tilde{B}_{i-1}&=&\frac{1}{2}\bigg[\sum_{j=1}^{i-2}N_j\int_{x_{i-3/2}}^{x_{i-1/2}}(x-x_{i-1})\beta(x-x_j,x_j)n(t,x-x_j)dx\\
&&+\int_{x_{i-1}}^{x_{i-1/2}}(x-x_{i-1})\beta(x-x_{i-1},x_{i-1})n(t,x-x_{i-1})n(t,x_{i-1})\Delta x_{i-1} dx].
\end{eqnarray*}
By setting $f(\cdot,y):=\beta(\cdot,y)n(t,\cdot)$, the above equation becomes
\begin{eqnarray}\label{1st eq f}
\tilde{V}_{i-1}-x_{i-1}\tilde{B}_{i-1}&=&\frac{1}{2}\bigg[\sum_{j=1}^{i-2}N_j\int_{x_{i-3/2}}^{x_{i-1/2}}(x-x_{i-1})f(x-x_j,x_j)dx\\
&&\nonumber+\int_{x_{i-1}}^{x_{i-1/2}}(x-x_{i-1})f(x-x_{i-1},x_{i-1})n(t,x_{i-1})\Delta x_{i-1} dx].
\end{eqnarray}
We use Taylor series expansions of each integrand about $x_{i-1}$ in equation (\ref{1st eq f}) as
\begin{eqnarray*}
(x-x_{i-1})f(x-x_j,x_j)&=&0+f(x_{i-1}-x_j,x_j)(x-x_{i-1})\\
&&+f_{x}(x_{i-1}-x_j,x_j)(x-x_{i-1})^2+{\mathcal O}(\Delta x^3),\\
(x-x_{i-1})f(x-x_{i-1},x_{i-1})&=&0+f(x_{i-1}-x_{i-1},x_{i-1})(x-x_{i-1})+{\mathcal O}(\Delta x^2).
\end{eqnarray*}
The substitution of the above Taylor series expansion in equation (\ref{1st eq f}) gives
\begin{eqnarray*}
\tilde{V}_{i-1}-x_{i-1}\tilde{B}_{i-1}&=&\frac{1}{2}\bigg[
\frac{1}{12}\sum_{j=1}^{i-2}N_j f_{x}(x_{i-1}-x_j,x_j)\Delta x_{i-1}^3\\
&&+\frac{1}{8}f(x_{i-1}-x_{i-1},x_{i-1})n(t,x_{i-1})\Delta x_{i-1}^3+{\mathcal O}(\Delta x^4)\bigg].
\end{eqnarray*}
Since $\beta(x_{i-1}-x_{i-1},x_{i-1})=\beta(0,x_{i-1})=0$, therefore we have $f(x_{i-1}-x_{i-1},x_{i-1})=0$. This implies that 
\begin{eqnarray*}
\tilde{V}_{i-1}-x_{i-1}\tilde{B}_{i-1}=\frac{1}{24}\sum_{j=1}^{i-2}N_j f_{x}(x_{i-1}-x_j,x_j)\Delta x_{i-1}^3 +{\mathcal O}(\Delta x^4).
\end{eqnarray*}
Again the application of Taylor series expansion gives us 
\begin{eqnarray}\label{final f}
\tilde{V}_{i-1}-x_{i-1}\tilde{B}_{i-1}=\frac{1}{24}\sum_{j=1}^{i-2}N_j f_{x}(x_{i}-x_j,x_j)\Delta x_{i-1}^3 +{\mathcal O}(\Delta x^4).
\end{eqnarray}
Finally, substituting (\ref{final f}) into (\ref{1st eq lambda}), we obtain
\begin{eqnarray}\label{2nd eq lambda}
 \lambda^{-}_{i}(\overline{v}_{i-1})\hat{B}_{i-1}&=&\frac{1}{12}\sum_{j=1}^{i-2}N_j f_{x}(x_{i}-x_j,x_j)\frac{\Delta x_{i-1}^3}{\Delta x_i + \Delta x_{i-1}}\\
&&-\sum_{j=1}^{i-2} N_j(t) \int_{x_{i-3/2}-x_j}^{x_{l_{i-1,j}+\frac{1}{2}\gamma_{i-1,j}}}\frac{(x'-x_{i-1}+x_j)}{\Delta x_i + \Delta x_{i-1}}f(x',x_j)dx'\nonumber\\
&&-\sum_{j=1}^{i-1}N_j(t)\int_{x_{l_{i,j}+\frac{1}{2}\gamma_{i,j}}}^{x_{i-1/2}-x_j}\frac{(x'-x_{i-1}+x_j)}{\Delta x_i + \Delta x_{i-1}}f(x',x_j)dx'\nonumber\\ 
&&-\frac{1}{12}\sum_{j=1}^{i-2}N_j(t)\sum_{k={l_{i-1,j}+\frac{1}{2}(\gamma_{i-1,j}+1)}}^{l_{i,j}+\frac{1}{2}(\gamma_{i,j}-1)}\frac{\Delta {x_k}^3}{\Delta x_i + \Delta x_{i-1}} 
f_{x'}(x_k,x_j)\nonumber\\
&& \nonumber+{\mathcal O}(\Delta x^3).
\end{eqnarray}
Next, the second term in (\ref{E:B_iCA}) is evaluated as 
\begin{eqnarray*}
 \lambda^{+}_i(\overline{v}_{i})\hat{B}_{i}&=&\frac{\overline{v}_{i}-x_{i+1}}{x_i-x_{i+1}}\hat{B}_{i}=\bigg(1-\frac{\overline{v}_{i}-x_{i}}{x_{i+1}-x_i}\bigg)\hat{B}_{i}\\
&=&\hat{B}_{i}-\frac{2}{\Delta x_{i+1} + \Delta x_{i}}(\hat{V}_{i}-x_{i}\hat{B}_{i}).
\end{eqnarray*}
Calculating as before, we estimate the above expression in the following form
 \begin{eqnarray}\label{3rd eq lambda}
 \lambda^{+}_{i}(\overline{v}_{i})\hat{B}_{i}&=&\hat{B}_{i}-\frac{1}{12}\sum_{j=1}^{i-1}N_j f_{x}(x_{i}-x_j,x_j)\frac{\Delta x_{i}^3}{\Delta x_i + \Delta x_{i+1}}\\
&&+\sum_{j=1}^{i-1} N_j(t) \int_{x_{i-1/2}-x_j}^{x_{l_{i,j}+\frac{1}{2}\gamma_{i,j}}}\frac{(x'-x_{i}+x_j)}{\Delta x_i + \Delta x_{i+1}}f(x',x_j)dx'\nonumber\\
&&+\sum_{j=1}^{i}N_j(t)\int_{x_{l_{i+1,j}+\frac{1}{2}\gamma_{i+1,j}}}^{x_{i+1/2}-x_j}\frac{(x'-x_{i}+x_j)}{\Delta x_i + \Delta x_{i+1}}f(x',x_j)dx'\nonumber\\ 
&&+\frac{1}{12}\sum_{j=1}^{i-1}N_j(t)\sum_{k={l_{i,j}+\frac{1}{2}(\gamma_{i,j}+1)}}^{l_{i+1,j}+\frac{1}{2}(\gamma_{i+1,j}-1)}\frac{\Delta {x_k}^3}{\Delta x_i + \Delta x_{i+1}} 
f_{x'}(x_k,x_j)\nonumber\\
&&\nonumber +{\mathcal O}(\Delta x^3).
\end{eqnarray}
Similar to the second term we obtain 
 \begin{eqnarray}\label{4th eq lambda}
 \lambda^{-}_{i}(\overline{v}_{i})\hat{B}_{i}&=&\hat{B}_{i}+\frac{1}{12}\sum_{j=1}^{i-1}N_j f_{x}(x_{i}-x_j,x_j)\frac{\Delta x_{i}^3}{\Delta x_i + \Delta x_{i-1}}\\
&&-\sum_{j=1}^{i-1} N_j(t) \int_{x_{i-1/2}-x_j}^{x_{l_{i,j}+\frac{1}{2}\gamma_{i,j}}}\frac{(x'-x_{i}+x_j)}{\Delta x_i + \Delta x_{i-1}}f(x',x_j)dx'\nonumber\\
&&-\sum_{j=1}^{i}N_j(t)\int_{x_{l_{i+1,j}+\frac{1}{2}\gamma_{i+1,j}}}^{x_{i+1/2}-x_j}\frac{(x'-x_{i}+x_j)}{\Delta x_i + \Delta x_{i-1}}f(x',x_j)dx'\nonumber\\ 
&&-\frac{1}{12}\sum_{j=1}^{i-1}N_j(t)\sum_{k={l_{i,j}+\frac{1}{2}(\gamma_{i,j}+1)}}^{l_{i+1,j}+\frac{1}{2}(\gamma_{i+1,j}-1)}\frac{\Delta {x_k}^3}{\Delta x_i + \Delta x_{i-1}} 
f_{x'}(x_k,x_j)\nonumber\\
&& \nonumber+{\mathcal O}(\Delta x^3).
\end{eqnarray}
Finally, similar to the first term we can easily estimate 
\begin{eqnarray}\label{5th eq lambda}
 \lambda^{+}_{i}(\overline{v}_{i+1})\hat{B}_{i+1}&=&\frac{1}{12}\sum_{j=1}^{i}N_j f_{x}(x_{i}-x_j,x_j)\frac{\Delta x_{i+1}^3}{\Delta x_i + \Delta x_{i+1}}\\
&&+\sum_{j=1}^{i} N_j(t) \int_{x_{i+1/2}-x_j}^{x_{l_{i+1,j}+\frac{1}{2}\gamma_{i+1,j}}}\frac{(x'-x_{i+1}+x_j)}{\Delta x_i + \Delta x_{i+1}}f(x',x_j)dx'\nonumber\\
&&+\sum_{j=1}^{i+1}N_j(t)\int_{x_{l_{i+2,j}+\frac{1}{2}\gamma_{i+2,j}}}^{x_{i+3/2}-x_j}\frac{(x'-x_{i+1}+x_j)}{\Delta x_i + \Delta x_{i+1}}f(x',x_j)dx'\nonumber\\ 
&&+\frac{1}{12}\sum_{j=1}^{i}N_j(t)\sum_{k={l_{i+1,j}+\frac{1}{2}(\gamma_{i+1,j}+1)}}^{l_{i+2,j}+\frac{1}{2}(\gamma_{i+2,j}-1)}\frac{\Delta {x_k}^3}{\Delta x_i + \Delta x_{i+1}} 
f_{x'}(x_k,x_j)\nonumber\\
&&\nonumber +{\mathcal O}(\Delta x^3).
\end{eqnarray}
By substituting (\ref{2nd eq lambda}), (\ref{3rd eq lambda}), (\ref{4th eq lambda}) and (\ref{5th eq lambda}) into 
(\ref{E:B_iCA}) and using (\ref{final B_i}), the local discretization error can be evaluated as follows

{\bf{Case I:}} $\overline{v}_{i-1}>x_{i-1}$, $\overline{v}_{i}>x_{i}$ and $\overline{v}_{i+1}\geq x_{i+1}$ 
\begin{eqnarray}\label{case1}
 \hat{B}^{CA}_i=B_i\hspace{3in}&&\\
\left.\begin{array}{lcr}
&&-\displaystyle\frac{1}{2}\sum_{j=1}^{i-1} N_j(t) \int_{x_{i-1/2}-x_j}^{x_{l_{i,j}+\frac{1}{2}\gamma_{i,j}}}f(x',x_j)dx'\nonumber\\
&&-\displaystyle\frac{1}{2}\sum_{j=1}^{i}N_j(t)\int_{x_{l_{i+1,j}+\frac{1}{2}\gamma_{i+1,j}}}^{x_{i+1/2}-x_j}f(x',x_j)dx'\phantom{\Bigg|}
\end{array}\right\}& =:E_1
\nonumber\\
\left.\begin{array}{rcl}
&&+\displaystyle\frac{1}{12}\bigg(\frac{\Delta x_{i-1}^3}{\Delta x_i + \Delta x_{i-1}}-\frac{\Delta x_{i}^3}{\Delta x_i + \Delta x_{i+1}}\bigg)\sum_{j=1}^{i}N_j f_{x'}(x_{i}-x_j,x_j)\phantom{\Bigg|}
\end{array}\right\}& =:E_2
\nonumber\\
\left.\begin{array}{lcr}
&&+\displaystyle\sum_{j=1}^{i-1} N_j(t) \int_{x_{i-1/2}-x_j}^{x_{l_{i,j}+\frac{1}{2}\gamma_{i,j}}}\frac{(x'-x_{i}+x_j)}{\Delta x_i + \Delta x_{i+1}}f(x',x_j)dx'\nonumber\\
&&+\displaystyle\sum_{j=1}^{i}N_j(t)\int_{x_{l_{i+1,j}+\frac{1}{2}\gamma_{i+1,j}}}^{x_{i+1/2}-x_j}\frac{(x'-x_{i}+x_j)}{\Delta x_i + \Delta x_{i+1}}f(x',x_j)dx'\nonumber\\ 
&&-\displaystyle\sum_{j=1}^{i-2} N_j(t) \int_{x_{i-3/2}-x_j}^{x_{l_{i-1,j}+\frac{1}{2}\gamma_{i-1,j}}}\frac{(x'-x_{i-1}+x_j)}{\Delta x_i + \Delta x_{i-1}}f(x',x_j)dx'\nonumber\\
&&-\displaystyle\sum_{j=1}^{i-1}N_j(t)\int_{x_{l_{i,j}+\frac{1}{2}\gamma_{i,j}}}^{x_{i-1/2}-x_j}\frac{(x'-x_{i-1}+x_j)}{\Delta x_i + \Delta x_{i-1}}f(x',x_j)dx'\phantom{\Bigg|}
\end{array}\right\}& =:E_3\nonumber
\end{eqnarray}
\begin{eqnarray}
\left.\begin{array}{rcl}
&&+\displaystyle\frac{1}{12}\sum_{j=1}^{i-1}N_j(t)\sum_{k={l_{i,j}+\frac{1}{2}(\gamma_{i,j}+1)}}^{l_{i+1,j}+\frac{1}{2}(\gamma_{i+1,j}-1)}\frac{\Delta {x_k}^3}{\Delta x_i + \Delta x_{i+1}} 
f_{x'}(x_k,x_j)\nonumber\\
&&-\displaystyle\frac{1}{12}\sum_{j=1}^{i-2}N_j(t)\sum_{k={l_{i-1,j}+\frac{1}{2}(\gamma_{i-1,j}+1)}}^{l_{i,j}+\frac{1}{2}(\gamma_{i,j}-1)}\frac{\Delta {x_k}^3}{\Delta x_i + \Delta x_{i-1}} 
f_{x'}(x_k,x_j)\phantom{\Bigg|}
\end{array}\right\}& =:E_4\nonumber\\ 
&& \nonumber\hspace{-6cm}+{\mathcal O}(\Delta x^3). 
\end{eqnarray}
{\bf{Case II:}} $\overline{v}_{i-1}\leq x_{i-1}$, $\overline{v}_{i}<x_{i}$ and $\overline{v}_{i+1}< x_{i+1}$ 

Similar to the previous case, we have 
\begin{eqnarray}\label{case2}
 \hat{B}^{CA}_i= B_i\hspace{3in}&&\\
\left.\begin{array}{lcr}
&&-\displaystyle\frac{1}{2}\sum_{j=1}^{i-1} N_j(t) \int_{x_{i-1/2}-x_j}^{x_{l_{i,j}+\frac{1}{2}\gamma_{i,j}}}f(x',x_j)dx'\nonumber\\
&&-\displaystyle\frac{1}{2}\sum_{j=1}^{i}N_j(t)\int_{x_{l_{i+1,j}+\frac{1}{2}\gamma_{i+1,j}}}^{x_{i+1/2}-x_j}f(x',x_j)dx'\phantom{\Bigg|}
\end{array}\right\}& =:E_1
\nonumber\\
\left.\begin{array}{rcl}
&&+\displaystyle\frac{1}{12}\bigg(\frac{\Delta x_{i}^3}{\Delta x_i + \Delta x_{i-1}}-\frac{\Delta x_{i+1}^3}{\Delta x_i + \Delta x_{i+1}}\bigg)\sum_{j=1}^{i}N_j f_{x'}(x_{i}-x_j,x_j)\phantom{\Bigg|}
\end{array}\right\}& =:E'_2
\nonumber\\
\left.\begin{array}{lcr}
&&+\displaystyle\sum_{j=1}^{i} N_j(t) \int_{x_{i+1/2}-x_j}^{x_{l_{i+1,j}+\frac{1}{2}\gamma_{i+1,j}}}\frac{(x'-x_{i+1}+x_j)}{\Delta x_i + \Delta x_{i+1}}f(x',x_j)dx'\nonumber\\
&&+\displaystyle\sum_{j=1}^{i+1}N_j(t)\int_{x_{l_{i+2,j}+\frac{1}{2}\gamma_{i+2,j}}}^{x_{i+3/2}-x_j}\frac{(x'-x_{i+1}+x_j)}{\Delta x_i + \Delta x_{i+1}}f(x',x_j)dx'\nonumber\\ 
&&-\displaystyle\sum_{j=1}^{i-1} N_j(t) \int_{x_{i-1/2}-x_j}^{x_{l_{i,j}+\frac{1}{2}\gamma_{i,j}}}\frac{(x'-x_{i}+x_j)}{\Delta x_i + \Delta x_{i-1}}f(x',x_j)dx'\nonumber\\
&&-\displaystyle\sum_{j=1}^{i}N_j(t)\int_{x_{l_{i+1,j}+\frac{1}{2}\gamma_{i+1,j}}}^{x_{i+1/2}-x_j}\frac{(x'-x_{i}+x_j)}{\Delta x_i + \Delta x_{i-1}}f(x',x_j)dx'\phantom{\Bigg|}
\end{array}\right\}& =:E'_3\nonumber\\ 
\left.\begin{array}{rcl}
&&+\displaystyle\frac{1}{12}\sum_{j=1}^{i}N_j(t)\sum_{k={l_{i+1,j}+\frac{1}{2}(\gamma_{i+1,j}+1)}}^{l_{i+2,j}+\frac{1}{2}(\gamma_{i+2,j}-1)}\frac{\Delta {x_k}^3}{\Delta x_i + \Delta x_{i+1}} 
f_{x'}(x_k,x_j)\nonumber\\
&&-\displaystyle\frac{1}{12}\sum_{j=1}^{i-1}N_j(t)\sum_{k={l_{i,j}+\frac{1}{2}(\gamma_{i,j}+1)}}^{l_{i+1,j}+\frac{1}{2}(\gamma_{i+1,j}-1)}\frac{\Delta {x_k}^3}{\Delta x_i + \Delta x_{i-1}} 
f_{x'}(x_k,x_j)\phantom{\Bigg|}
\end{array}\right\}& =:E'_4\nonumber\\ 
&&\nonumber \hspace{-6cm}+{\mathcal O}(\Delta x^3). 
\end{eqnarray}
{\bf{Case III:}} $\overline{v}_{i-1}\leq x_{i-1}$, $\overline{v}_{i}=x_{i}$ and $\overline{v}_{i+1}\geq x_{i+1}$ 
\begin{eqnarray}\label{case3}
 \hat{B}^{CA}_i= B_i\hspace{3in}&&\\
\left.\begin{array}{lcr}
&&-\displaystyle\frac{1}{2}\sum_{j=1}^{i-1} N_j(t) \int_{x_{i-1/2}-x_j}^{x_{l_{i,j}+\frac{1}{2}\gamma_{i,j}}}f(x',x_j)dx'\nonumber\\
&&-\displaystyle\frac{1}{2}\sum_{j=1}^{i}N_j(t)\int_{x_{l_{i+1,j}+\frac{1}{2}\gamma_{i+1,j}}}^{x_{i+1/2}-x_j}f(x',x_j)dx'\phantom{\Bigg|}
\end{array}\right\}& =:E_1
\nonumber\\
&& \nonumber\hspace{-6cm}+{\mathcal O}(\Delta x^3). 
\end{eqnarray}

Next, the discretization error for death term is calculated in the $i$th cell. From equation (\ref{E:D_i}), the integrated death term can be written as follows
\begin{eqnarray*}
D_i = \int_{x_{i-1/2}}^{x_{i+1/2}}\sum_{j=1}^{I}\int_{x_{j-1/2}}^{x_{j+1/2}}K(x, y)n(t,y)n(t,x)dy dx.
\end{eqnarray*}
The application of the midpoint rule to the outer and inner integrals gives us
%
\begin{eqnarray}\label{death term for i}
D_i = N_i(t)\sum_{j=1}^{I}K(x_i, x_j)N_j(t) + {\mathcal O}(\Delta x^3)=\hat{D}_i + {\mathcal O}(\Delta x^3).
\end{eqnarray}
From the equations (\ref{case1}-\ref{case3}) and (\ref{death term for i}), we can
estimate the local discretization error $\sigma _i(t)=(B_i-D_i)-(\hat{B}^{CA}_i-\hat{B}^{CA}_i)$ as
\begin{eqnarray}\label{local error}
\sigma_i(t) = 
\begin{cases}
E_1+E_2+E_3+E_4+{\mathcal O}(\Delta x^3)\ \, & \text{if}\ i \in \mathfrak{A_1},\\
 E_1+E'_2+E'_3+E'_4+{\mathcal O}(\Delta x^3)\ \, & \text{if}\ i \in \mathfrak{A_2},\\
E_1+{\mathcal O}(\Delta x^3)\ \, & \text{if}\ i \in \mathfrak{A_3}.
\end{cases}
\end{eqnarray}
where
\begin{eqnarray*}
\mathfrak{A_1} &= \{i \in \mathbb{N}\, | \,\bar{v}_{i-1}>x_{i-1}, \bar{v}_{i}>x_{i}, \bar{v}_{i+1}\geq x_{i+1}\}, \\
\mathfrak{A_2} &= \{i \in \mathbb{N}\, | \,\bar{v}_{i-1}\leq x_{i-1},\bar{v}_{i}<x_{i}, \bar{v}_{i+1}<x_{i+1}\}, \\
\mathfrak{A_3} &= \{i \in \mathbb{N}\, | \,\bar{v}_{i-1} \leq x_{i-1},\bar{v}_{i}=x_{i}, \bar{v}_{i+1}\geq x_{i+1}\}.
\end{eqnarray*}
Here we consider three different cases to find the order of consistency. Then, the order of consistency is given by
\begin{eqnarray}\label{consistency order}
\|\sigma(t)\| =  \sum_{i\in \mathfrak{A_1}}|\sigma_i(t)| + \sum_{i\in \mathfrak{A_2}}|\sigma_i(t)| + \sum_{i\in \mathfrak{A_3}}|\sigma_i(t)|.
\end{eqnarray}
The following types of grids will be considered to find the order of consistency of CAT. 

\subsection{Uniform grids} 

Let us begin with the case of uniform grids i.e.
\ $\Delta x_i = \Delta x \ \ \mbox{and}\ \ x_i=(i-1/2)\Delta x\ \ \mbox{for any}\ \ i=1, \ldots, I$. Here, $E_2$ and $E'_2$ defined in 
(\ref{case1}) and (\ref{case2}), respectively, are obviously zero. In case of such uniform grids, we have
$$x_{i-1/2}-x_j= x_{i-j},\hspace{.1in} x_{i+1/2}-x_j= x_{i-j+1},\hspace{.1in} \mbox{and} \hspace{.1in}x_{i-3/2}-x_j= x_{i-j-1}.$$
By using the definition of indices $l_{i,j}$ and $\gamma_{i,j}$ in (\ref{defi of index}), we calculate 
$$x_{i-1/2}-x_j= x_{i-j} \in \Lambda_{l_{i,j}},$$
which gives 
$$x_{i-1/2}-x_j= x_{i-j}=x_{l_{i,j}}.$$
Similarly, we obtain  
$$x_{i+1/2}-x_j= x_{i-j+1}=x_{l_{i+1,j}},$$
and 
$$x_{i-3/2}-x_j= x_{i-j-1}=x_{l_{i-1,j}}.$$
This shows that $\gamma_{i-1,j}= \gamma_{i,j}=\gamma_{i+1,j}=0$. Therefore, in (\ref{case1})-(\ref{case2}), the error terms $E_1$, $E_3$ and $E'_3$ become zero. It can also 
be easily realized that $x_{i-3/2}-x_j$, $x_{i-1/2}-x_j$, and $x_{i+1/2}-x_j$ are the pivot points 
of the adjacent cells, i.e.\ $l_{i-1,j}=(i-j-1)$th, $l_{i,j}=(i-j)$th and $l_{i+1,j}=(i-j+1)$th cells, respectively. Thus, by 
substituting the values of all these indices in $E_4$ and $E'_4$ defined in (\ref{case1}) and (\ref{case2}), respectively, and applying the Taylor series expansion, we obtain 
$E_4={\mathcal O}(\Delta x^3)$ and $E'_4={\mathcal O}(\Delta x^3)$. Then, from (\ref{local error}), we obtain
\begin{eqnarray*}
\sigma_i(t) = {\mathcal O}(\Delta x^3)\ \,  \text{if}\ \ i\in \mathfrak{A_1}, \mathfrak{A_2}, \mathfrak{A_3}.
\end{eqnarray*}
By using (\ref{consistency order}), the order of consistency is thus given by 
\begin{eqnarray*}
\|\sigma(t)\| ={\mathcal O}(\Delta x^2).
\end{eqnarray*}
Therefore, the cell average technique is second order consistent on uniform grids.

\subsection{Non-uniform smooth grids}
Non-uniform smooth grids can be obtained by applying some smooth transformation to uniform grids. Assume a variable $\xi$ 
with uniform grids and a smooth transformation $x=g(\xi)$ such that $x_{i\pm \frac{1}{2}} =g(\xi_{i\pm \frac{1}{2}})$ for any 
$i=1,\ldots, I$ to get non-uniform smooth grids. In this case, we show that the scheme is first order consistent. Let $h$ be 
the uniform mesh width in the variable $\xi$. For such type of smooth grids, Taylor series expansions in smooth 
transformations give 
$$\Delta x _i=x_{i+\frac{1}{2}}-x_{i-\frac{1}{2}}=g(\xi_i + \frac{h}{2})-g(\xi_i-\frac{h}{2})=hg'(\xi_i)+ \mathcal{O}(h^3).$$
Hence, by calculating $\Delta x_{i-1}$ and $\Delta x_{i+1}$ similar to $\Delta x_i$, we obtain
$$\Delta x_i -\Delta x_{i-1}=\mathcal{O}(h^2),$$
$$\Delta x_i +\Delta x_{i+1}=h[g'(\xi_i)+g'(\xi_{i+1})]+\mathcal{O}(h^3)=2hg'(\xi_i)+\mathcal{O}(h^2),$$
and similarly, we have
$$\Delta x_i +\Delta x_{i-1}=h[g'(\xi_i)+g'(\xi_{i-1})]+\mathcal{O}(h^3)=2hg'(\xi_i)+\mathcal{O}(h^2).$$
In particular, we deal with a special type of non-uniform smooth grids  which is known as {\bf\textit{geometric grids}}. Such type of 
grids can be defined as $x_{i+\frac{1}{2}}=rx_{i-\frac{1}{2}}$, $r>1$, $i=1,\ldots, I$. An exponential function 
can be applied on uniform grids as a smooth transformation to construct such type of geometric grids. Mathematically, we write
\begin{eqnarray*} 
x_{i+1/2} = \exp(\xi_{i+1/2})& = &\exp(h +\xi_{i-1/2}) \\
&=& \exp(h)\exp(\xi_{i-1/2})\\
&=&\exp(h) x_{i-1/2} =: r x_{i-1/2}, \, r>1.
\end{eqnarray*}
To solve the error terms appearing in (\ref{local error}), let us further assume that $\xi_{11}$, $\xi_{12}$, $\xi_{21}$, $\xi_{22}$, $\xi_{31}$ and $\xi_{32}$ are corresponding
points on uniform mesh for $x_{l_{i+1,j+1}+ \frac{1}{2}\gamma_{i+1,j+1}}$, $x_{i+1/2}-x_{j+1}$,
$x_{l_{i,j}+ \frac{1}{2}\gamma_{i,j}}$, $x_{i-1/2}-x_{j}$, $x_{l_{i-1,j-1}+ \frac{1}{2}\gamma_{i-1,j-1}}$ and $x_{i-3/2}-x_{j-1}$, respectively. Due to an application of exponential 
smooth transformation, these points can be defined as
$$\xi_{11}=\mbox{ln}(x_{l_{i+1,j+1}+\frac{1}{2}\gamma_{i+1,j+1}}), \ldots, \xi_{32}=\mbox{ln}(x_{i-3/2}-x_{j-1}).$$
By the definition of the indices in (\ref{defi of index}), we know
$$x_{i+1/2}-x_{j+1} \in \Lambda_{l_{i+1,j+1}}, \ \ x_{i-1/2}-x_{j} \in \Lambda_{l_{i,j}}\ \ \mbox{and} \ \ x_{i-3/2}-x_{j-1} \in \Lambda_{l_{i-1,j-1}}.$$
For geometric grids, we have
$$x_{i+1/2}-x_{j+1}=r(x_{i-1/2}-x_{j})=r^2(x_{i-3/2}-x_{j-1}).$$
Therefore, we have
$$l_{i+1,j+1}=l_{i,j}+1=l_{i-1,j-1}+2.$$
Further, in case of geometric grids, we have
$$\gamma_{i+1,j+1}=\gamma_{i,j}=\gamma_{i-1,j-1}.$$
Let us consider
\begin{eqnarray*}
 h_1=\xi_{11}-\xi_{12}&=&\mbox{ln}(x_{l_{i+1,j+1}+\frac{1}{2}\gamma_{i+1,j+1}})-\mbox{ln}(x_{i+1/2}-x_{j+1})\\
&&=\mbox{ln}\bigg(\frac{x_{l_{i+1,j+1}+\frac{1}{2}\gamma_{i+1,j+1}}}{x_{i+1/2}-x_{j+1}}\bigg)=\mbox{ln}\bigg(\frac{x_{l_{i,j}+\frac{1}{2}\gamma_{i,j}}}{x_{i-1/2}-x_{j}}\bigg)=\xi_{21}-\xi_{22}\\
&&=\mbox{ln}\bigg(\frac{x_{l_{i-1,j-1}+\frac{1}{2}\gamma_{i-1,j-1}}}{x_{i-3/2}-x_{j-1}}\bigg)=\xi_{31}-\xi_{32}.
\end{eqnarray*}
Similarly, we estimate
\begin{eqnarray*}
 \xi_{12}-\xi_{22}= \mbox{ln}\bigg(\frac{x_{i+1/2}-x_{j+1}}{x_{i-1/2}-x_{j}}\bigg)=\mbox{ln}(r)=h,
\end{eqnarray*}
and 
\begin{eqnarray*}
 \xi_{22}-\xi_{32}= \mbox{ln}\bigg(\frac{x_{i-1/2}-x_{j}}{x_{i-3/2}-x_{j-1}}\bigg)=\mbox{ln}(r)=h.
\end{eqnarray*}
Again, by application of smooth transformation, we can easily obtain 
\begin{eqnarray}\label{smooth12}
x_{l_{i+1,j+1}+\frac{1}{2}\gamma_{i+1,j+1}}-(x_{i+1/2}-x_{j+1})=g(\xi_{11})-g(\xi_{12})=h_1g'(\xi_{12}) + {\mathcal O}(h^2),
\end{eqnarray}
\begin{eqnarray}\label{smooth22}
 x_{l_{i,j}+\frac{1}{2}\gamma_{i,j}}-(x_{i-1/2}-x_{j})=g(\xi_{21})-g(\xi_{22})=h_1g'(\xi_{22}) + {\mathcal O}(h^2),
\end{eqnarray}
and
\begin{eqnarray}\label{smooth32}
x_{l_{i-1,j-1}+\frac{1}{2}\gamma_{i-1,j-1}}-(x_{i-3/2}-x_{j-1})=g(\xi_{31})-g(\xi_{32})=h_1g'(\xi_{32}) + {\mathcal O}(h^2).
\end{eqnarray}
All these identities will play an important role to solve the error terms involved in (\ref{local error}), which helps us 
to calculate the order of local discretization error $\sigma_i$.

We first evaluate $E_1$ as follows
\begin{eqnarray*}
 E_1=\frac{1}{2}\sum_{j=1}^{i-1} N_j(t) \int_{x_{i-1/2}-x_j}^{x_{l_{i,j}+\frac{1}{2}\gamma_{i,j}}}f(x,x_j)dx+\frac{1}{2}\sum_{j=1}^{i}N_j(t)\int_{x_{l_{i+1,j}+\frac{1}{2}\gamma_{i+1,j}}}^{x_{i+1/2}-x_j}f(x,x_j)dx.
\end{eqnarray*}
Applying the left and right rectangle rules in the integrals involved in the first and second terms, respectively, on the right-hand side, we estimate 
\begin{eqnarray*}
 E_1&=&\frac{1}{2}\sum_{j=1}^{i-1} N_j(t) f(x_{i-1/2}-x_j, x_j)(x_{l_{i,j}+\frac{1}{2}\gamma_{i,j}}-x_{i-1/2}+x_{j})\\
&&+\frac{1}{2}\sum_{j=1}^{i}N_j(t)f(x_{i+1/2}-x_j, x_j)(x_{i+1/2}-x_{j}-x_{l_{i+1,j}+\frac{1}{2}\gamma_{i+1,j}}) + \mathcal{O}(\Delta x^2).
\end{eqnarray*}
Then an application of Taylor's series expansion about $x_{i-1/2}=x_{i+1/2}$ in $f(x_{i-1/2}-x_{j})$ gives 
 \begin{eqnarray*}
 E_1&=&\frac{1}{2}\sum_{j=1}^{i-1} N_j(t) f(x_{i+1/2}-x_j, x_j)(x_{l_{i,j}+\frac{1}{2}\gamma_{i,j}}-x_{i-1/2}+x_{j})\\
&&-\frac{1}{2}\sum_{j=1}^{i}N_j(t)f(x_{i+1/2}-x_j, x_j)(x_{l_{i+1,j}+\frac{1}{2}\gamma_{i+1,j}}-x_{i+1/2}+x_{j}) + \mathcal{O}(\Delta x^2).
\end{eqnarray*}
We replace $j$ by $j+1$ in the second term on the right-hand side and use the relationship $N_j(t)=n(t, x_j)\Delta x_j + \mathcal{O}(\Delta x^3)$ for the 
midpoint rule. Also, we drop the term which is of second order, and obtain 
 \begin{eqnarray*}
 E_1&=&\frac{1}{2}\sum_{j=1}^{i-1} n(t, x_j)\Delta x_j  f(x_{i+1/2}-x_j, x_j)(x_{l_{i,j}+\frac{1}{2}\gamma_{i,j}}-x_{i-1/2}+x_{j})\\
&&-\frac{1}{2}\sum_{j=1}^{i-1}n(t, x_{j+1})\Delta x_{j+1} f(x_{i+1/2}-x_{j+1}, x_{j+1})(x_{l_{i+1,j+1}+\frac{1}{2}\gamma_{i+1,j+1}}-x_{i+1/2}+x_{j+1})\\ 
&& + \mathcal{O}(\Delta x^2).
\end{eqnarray*}
Approximating the function $x \mapsto n(t,x)f(x_{i\pm1/2}-x, x)$ at $x_j$ by $n(t,x)f(x_{i\pm1/2}-x, x)$ evaluated at $x=x_{j+1}$ in the first term, we evaluate 
\begin{eqnarray}\label{E1}
 \hspace{.4in}E_1&=&\frac{1}{2}\sum_{j=1}^{i-1}\{\Delta x_j(x_{l_{i,j}+\frac{1}{2}\gamma_{i,j}}-x_{i-1/2}+x_{j})-\Delta x_{j+1} (x_{l_{i+1,j+1}+\frac{1}{2}\gamma_{i+1,j+1}}-x_{i+1/2}+x_{j+1})\}\\
&&\nonumber \hspace{.6in}\times n(t, x_{j+1})f(x_{i+1/2}-x_{j+1}, x_{j+1}) + \mathcal{O}(\Delta x^2).
\end{eqnarray}
By using the identities in the beginning of this section, we calculate
\begin{eqnarray}\label{cal E1}
 \hspace{.4in}\Delta x_j&&(x_{l_{i,j}+\frac{1}{2}\gamma_{i,j}}-x_{i-1/2}+x_{j})-\Delta x_{j+1} (x_{l_{i+1,j+1}+\frac{1}{2}\gamma_{i+1,j+1}}-x_{i+1/2}+x_{j+1})\\
= &&(\Delta x_j-\Delta x_{j+1})(x_{l_{i,j}+\frac{1}{2}\gamma_{i,j}}-x_{i-1/2}+x_{j})\nonumber\\ 
&&-\Delta x_{j+1} [(x_{l_{i+1,j+1}+\frac{1}{2}\gamma_{i+1,j+1}}-x_{i+1/2}+x_{j+1})-(x_{l_{i,j}+\frac{1}{2}\gamma_{i,j}}-x_{i-1/2}+x_{j})]\nonumber\\
=&&\mathcal{O}(h^2)\{h_1g'(\xi_{22}) + {\mathcal O}(h^2)\}-\{hg'(\xi_{j+1})+\mathcal {O}(h^3)\}[h_1\{g'(\xi_{12})-g'(\xi_{22})\}+ {\mathcal O}(h^2)]\nonumber\\
=&& \mathcal {O}(h^3)-\mathcal{O}(h)[h_1hg'(\xi_{22})+ {\mathcal O}(h^2)]\nonumber\\
 \nonumber=&&\mathcal {O}(h^3).
\end{eqnarray}
Therefore, substituting (\ref{cal E1}) in (\ref{E1}), we obtain 
\begin{eqnarray}\label{final E1}
 E_1= {\mathcal O}(\Delta x^2).
\end{eqnarray}
Next, let us calculate $E_3$ defined in (\ref{case1}) as follows
\begin{eqnarray*}
E_3&=&\sum_{j=1}^{i-1} N_j(t) \int_{x_{i-1/2}-x_j}^{x_{l_{i,j}+\frac{1}{2}\gamma_{i,j}}}\frac{(x-x_{i}+x_j)}{\Delta x_i + \Delta x_{i+1}}f(x,x_j)dx\nonumber\\
&&+\sum_{j=1}^{i}N_j(t)\int_{x_{l_{i+1,j}+\frac{1}{2}\gamma_{i+1,j}}}^{x_{i+1/2}-x_j}\frac{(x-x_{i}+x_j)}{\Delta x_i + \Delta x_{i+1}}f(x,x_j)dx\nonumber\\ 
&&-\sum_{j=1}^{i-2} N_j(t) \int_{x_{i-3/2}-x_j}^{x_{l_{i-1,j}+\frac{1}{2}\gamma_{i-1,j}}}\frac{(x-x_{i-1}+x_j)}{\Delta x_i + \Delta x_{i-1}}f(x,x_j)dx\nonumber\\
&&-\sum_{j=1}^{i-1}N_j(t)\int_{x_{l_{i,j}+\frac{1}{2}\gamma_{i,j}}}^{x_{i-1/2}-x_j}\frac{(x-x_{i-1}+x_j)}{\Delta x_i + \Delta x_{i-1}}f(x,x_j)dx.
\end{eqnarray*}
Applying the left rectangle rule to the integrals appearing in first and third terms, and the right rectangle rule to the integrals in second and fourth terms, we estimate 
\begin{eqnarray*}
 E_3&=&-\frac{1}{2}\sum_{j=1}^{i-1} N_j(t)f(x_{i-1/2}-x_j,x_j)\frac{\Delta x_i }{\Delta x_i + \Delta x_{i+1}}(x_{l_{i,j}+\frac{1}{2}\gamma_{i,j}}-x_{i-1/2}+x_{j})\\
&&-\frac{1}{2}\sum_{j=1}^{i-1} N_j(t)f(x_{i+1/2}-x_j,x_j)\frac{\Delta x_i }{\Delta x_i + \Delta x_{i+1}}(x_{l_{i+1,j}+\frac{1}{2}\gamma_{i+1,j}}-x_{i+1/2}+x_{j})\\
&&+\frac{1}{2}\sum_{j=1}^{i-2} N_j(t)f(x_{i-3/2}-x_j,x_j)\frac{\Delta x_{i-1} }{\Delta x_i + \Delta x_{i-1}}(x_{l_{i-1,j}+\frac{1}{2}\gamma_{i-1,j}}-x_{i-3/2}+x_{j})\\
&&+\frac{1}{2}\sum_{j=1}^{i-1} N_j(t)f(x_{i-1/2}-x_j,x_j)\frac{\Delta x_{i-1} }{\Delta x_i +x_{i-1}}(x_{l_{i,j}+\frac{1}{2}\gamma_{i,j}}-x_{i-1/2}+x_{j}) + \mathcal{O}(\Delta x^2).
\end{eqnarray*}
Let us approximate $f$ at $(x_{i-3/2}-x_j,x_j)$ by $f$ expanded around $(x_{i-1/2}-x_j,x_j)$ in the third term and $f$ at $(x_{i-1/2}-x_j,x_j)$ by $f$ expanded around $(x_{i+1/2}-x_j,x_j)$ 
in the fourth term. Further, we replace $j$ by $j+1$ and $j-1$ respectively in second and third terms. Also, the relationship $N_j(t)=n(t,x_j)\Delta x_j + \mathcal{O}(\Delta x^3)$ is used to get 
\begin{eqnarray*}
 E_3&=&-\frac{1}{2}\sum_{j=1}^{i-1} n(t,x_j)\Delta x_j f(x_{i-1/2}-x_j,x_j)\frac{\Delta x_i }{\Delta x_i + \Delta x_{i+1}}(x_{l_{i,j}+\frac{1}{2}\gamma_{i,j}}-x_{i-1/2}+x_{j})\\
&&-\frac{1}{2}\sum_{j=0}^{i-2} n(t,x_{j+1})\Delta x_{j+1} f(x_{i+1/2}-x_{j+1},x_{j+1})\\
&&\hspace{.2in}\times \frac{\Delta x_i }{\Delta x_i + \Delta x_{i+1}}(x_{l_{i+1,j+1}+\frac{1}{2}\gamma_{i+1,j+1}}-x_{i+1/2}+x_{j+1})\\
&&+\frac{1}{2}\sum_{j=2}^{i-1} n(t,x_{j-1})\Delta x_{j-1} f(x_{i-1/2}-x_{j-1},x_{j-1})\\
&&\hspace{.2in}\times\frac{\Delta x_{i-1} }{\Delta x_i + \Delta x_{i-1}}(x_{l_{i-1,j-1}+\frac{1}{2}\gamma_{i-1,j-1}}-x_{i-3/2}+x_{j-1})\\
&&+\frac{1}{2}\sum_{j=1}^{i-2} n(t,x_j)\Delta x_j f(x_{i+1/2}-x_j,x_j)\frac{\Delta x_{i-1} }{\Delta x_i +x_{i-1}}(x_{l_{i,j}+\frac{1}{2}\gamma_{i,j}}-x_{i-1/2}+x_{j})\\
&& +\mathcal{O}(\Delta x^2).
\end{eqnarray*}
Without loss of generality, we can drop the terms which are second order accurate. Moreover, we approximate the functions $x \mapsto n(t,x)f(x_{i\mp1/2}-x,x)$ at point $x_j$ by 
$n(t,x)f(x_{i\mp1/2}-x,x)$ evaluated at points $x=x_{j\mp 1}$ of the first  and fourth terms, respectively, to obtain
\begin{eqnarray*}
 E_3&=&-\frac{1}{2}\sum_{j=2}^{i-1} n(t,x_{j-1}) f(x_{i-1/2}-x_{j-1},x_{j-1})\frac{\Delta x_i \Delta x_j }{\Delta x_i + \Delta x_{i+1}}(x_{l_{i,j}+\frac{1}{2}\gamma_{i,j}}-x_{i-1/2}+x_{j})\\
&&-\frac{1}{2}\sum_{j=1}^{i-2} n(t,x_{j+1}) f(x_{i+1/2}-x_{j+1},x_{j+1})\\
&&\hspace{.2in}\times \frac{\Delta x_i \Delta x_{j+1}}{\Delta x_i + \Delta x_{i+1}}(x_{l_{i+1,j+1}+\frac{1}{2}\gamma_{i+1,j+1}}-x_{i+1/2}+x_{j+1})\\
&&+\frac{1}{2}\sum_{j=2}^{i-1} n(t,x_{j-1}) f(x_{i-1/2}-x_{j-1},x_{j-1})\\
&&\hspace{.2in}\times\frac{\Delta x_{i-1} \Delta x_{j-1}}{\Delta x_i + \Delta x_{i-1}}(x_{l_{i-1,j-1}+\frac{1}{2}\gamma_{i-1,j-1}}-x_{i-3/2}+x_{j-1})\\
&&+\frac{1}{2}\sum_{j=1}^{i-2} n(t,x_{j+1}) f(x_{i+1/2}-x_{j+1},x_{j+1})\frac{\Delta x_{i-1} \Delta x_j}{\Delta x_i +x_{i-1}}(x_{l_{i,j}+\frac{1}{2}\gamma_{i,j}}-x_{i-1/2}+x_{j})\\
 &&+ \mathcal{O}(\Delta x^2).
\end{eqnarray*}
Let us denote each summation with the factor $\frac{1}{2}$ on the right-hand side by $E_{11},\ldots, E_{14}$ respectively. Therefore, we can write
\begin{eqnarray}\label{E3}
 E_3=(E_{33}-E_{31})+ (E_{34}-E_{32}) + \mathcal{O}(\Delta x^2).
\end{eqnarray}
To simplify (\ref{E3}), we first calculate $E_{33}-E_{31}$ as follows
\begin{eqnarray}\label{E331}
 E_{33}-E_{31}&=&\frac{1}{2}\sum_{j=2}^{i-1} n(t,x_{j-1}) f(x_{i-1/2}-x_{j-1},x_{j-1})\\
&&\hspace{.2in}\times\bigg\{\frac{\Delta x_{i-1} \Delta x_{j-1}}{\Delta x_i + \Delta x_{i-1}}(x_{l_{i-1,j-1}+\frac{1}{2}\gamma_{i-1,j-1}}-x_{i-3/2}+x_{j-1})\nonumber\\
&&\nonumber\hspace{.4in}- \frac{\Delta x_i \Delta x_j }{\Delta x_i + \Delta x_{i+1}}(x_{l_{i,j}+\frac{1}{2}\gamma_{i,j}}-x_{i-1/2}+x_{j})\bigg\}.
\end{eqnarray}
Again by using the identities mentioned in the beginning of this section, we need to estimate the following term 
for solving (\ref{E331}). 
\begin{eqnarray*}
&&\Delta x_{i-1} \Delta x_{j-1}(\Delta x_i + \Delta x_{i+1})(x_{l_{i-1,j-1}+\frac{1}{2}\gamma_{i-1,j-1}}-x_{i-3/2}+x_{j-1})\nonumber\\
&&\hspace{.2in}- \Delta x_i \Delta x_j(\Delta x_i + \Delta x_{i-1}) (x_{l_{i,j}+\frac{1}{2}\gamma_{i,j}}-x_{i-1/2}+x_{j})\\
&=&\Delta x_{i-1} \underset{=\mathcal{O}(h^2)}{\underbrace{(\Delta x_{j-1}-\Delta x_j)}}(\Delta x_i + \Delta x_{i+1})(x_{l_{i-1,j-1}+\frac{1}{2}\gamma_{i-1,j-1}}-x_{i-3/2}+x_{j-1})\\
&&+ \Delta x_j [\Delta x_{i-1}(\Delta x_i + \Delta x_{i+1})(x_{l_{i-1,j-1}+\frac{1}{2}\gamma_{i-1,j-1}}-x_{i-3/2}+x_{j-1})\\
&&\hspace{.2in}- \Delta x_i (\Delta x_i + \Delta x_{i-1}) (x_{l_{i,j}+\frac{1}{2}\gamma_{i,j}}-x_{i-1/2}+x_{j})]\\
&=& \mathcal{O}(h^5) + (hg'(\xi_j)+ \mathcal{O}(h^3))[(\underset{=hg'(\xi_i)+\mathcal{O}(h^2)}{\underbrace{hg'(\xi_{i-1})}}+ \mathcal{O}(h^3))(2hg'(\xi_i)+ \mathcal{O}(h^2))(h_1g'(\xi_{32})+ \mathcal{O}(h^2))\\
&&\hspace{1in}- (hg'(\xi_{i})+ \mathcal{O}(h^3))(2hg'(\xi_{i})+ \mathcal{O}(h^2))(h_1g'(\xi_{22})+ \mathcal{O}(h^2))]\\
&=& \mathcal{O}(h)[2h^2(g'(\xi_i))^2 h_1\{g'(\xi_{32})-g'(\xi_{22})\}]+ \mathcal{O}(h^5)\\
&=& \mathcal{O}(h)\cdot 2h^3h_1(g'(\xi_i))^2 g'(\xi_{32}) + \mathcal{O}(h^5)\\
&=& \mathcal{O}(h^5).
\end{eqnarray*}
Inserting this estimate in (\ref{E331}), we obtain 
\begin{eqnarray}\label{E331fin}
  E_{33}-E_{31}= \mathcal{O}(h^2).
\end{eqnarray}
Analogous to (\ref{E331fin}), we can easily show that
\begin{eqnarray}\label{E342}
 E_{34}-E_{32}= \mathcal{O}(h^2).
\end{eqnarray}
Finally, substituting (\ref{E331fin}) and (\ref{E342}) into (\ref{E3}), we have
\begin{eqnarray}\label{E3final}
 E_{3}= \mathcal{O}(\Delta x^2).
\end{eqnarray}
In a similar way, we can prove that
\begin{eqnarray}\label{E3final1}
 E'_{3}= \mathcal{O}(\Delta x^2).
\end{eqnarray}
Next, it can be easily observed from (\ref{case1}) and (\ref{case2}) that the error terms $E_2$, $E'_2$, $E_4$ and $E'_4$ are second order accurate independent of meshes. Therefore, by 
substituting (\ref{final E1}), (\ref{E3final}) and (\ref{E3final1}) into (\ref{local error}), we have
\begin{eqnarray*}
 \sigma_i(t)=\mathcal{O}(\Delta x^2) \ \ \mbox{if}\ \ i\in \mathfrak{A_1}, \mathfrak{A_2}, \mathfrak{A_3}.
\end{eqnarray*}
Thus, using (\ref{consistency order}), we obtain 
\begin{eqnarray*}
 \|\sigma(t)\|=\mathcal{O}(\Delta x).
\end{eqnarray*}
This shows that the cell average technique is first order consistent on such type of non-uniform smooth grids.
\begin{remark}
It should be pointed out that, due to the cancellation of second order terms, the error terms $E_2$, $E'_2$, $E_3$, $E'_3$, $E_4$ and $E'_4$ 
can be shown third order accurate on geometric grids. However, since this will not improve the order of consistency (because $E_1$ is only 
second order accurate for such grids), we do not include further calculations.
\end{remark}

\section{Lipschitz conditions on $\hat{\mathbf{B}}({\mathbf{N}}(t))$ and $\hat{\mathbf{D}}({\mathbf{N}}(t))$}\label{convergence}
Let us consider the birth term for $0 \leq t \leq T$ and for all $\mathbf{N}$, $\hat{\mathbf{N}}\in \mathbb{R}^I$. We get from (\ref{E:B_iCA})
\begin{eqnarray*}
 \|\hat{\mathbf{B}}({\mathbf{N}})-\hat{\mathbf{B}}({\hat{\mathbf{N}}})\|\leq &&\sum_{i=1}^{I}\lambda^{-}_i(\overline{v}_{i-1})H(\overline{v}_{i-1}-x_{i-1})|\hat{B}_{i-1}(\mathbf{N})-\hat{B}_{i-1}(\hat{\mathbf{N}})|\\
 &&+ \sum_{i=1}^{I}[\lambda^{+}_i(\overline{v}_{i})H(\overline{v}_{i}-x_{i})+\lambda^{-}_i(\overline{v}_{i})H(x_{i}-\overline{v}_{i})]|\hat{B}_{i}(\mathbf{N})-\hat{B}_{i}(\hat{\mathbf{N}})|\nonumber\\
 &&+ \sum_{i=1}^{I}\lambda^{+}_i(\overline{v}_{i+1})H(x_{i+1}-\overline{v}_{i+1})|\hat{B}_{i+1}(\mathbf{N})-\hat{B}_{i+1}(\hat{\mathbf{N}})|.
\end{eqnarray*}
 The definitions of $\lambda_i^{\pm}(x)$ and $H(x)$ in (\ref{values of lambda}) and (\ref{heavyside}), respectively, guarantee that $0\leq \lambda_i^{\pm}(x)H(x) \leq 1$. Thus, by using this upper bound, the above inequality becomes
 \begin{eqnarray}\label{Bcap}
\|\hat{\mathbf{B}}({\mathbf{N}})-\hat{\mathbf{B}}({\hat{\mathbf{N}}})\|&\leq &\sum_{i=1}^{I}|\hat{B}_{i-1}(\mathbf{N})-\hat{B}_{i-1}(\hat{\mathbf{N}})|+ \sum_{i=1}^{I}|\hat{B}_{i}(\mathbf{N})-\hat{B}_{i}(\hat{\mathbf{N}})|\\
 &&\nonumber+\sum_{i=1}^{I} |\hat{B}_{i+1}(\mathbf{N})-\hat{B}_{i+1}(\hat{\mathbf{N}})|.
\end{eqnarray}
By (\ref{cond on beta}), there exists a constant $C>0$ such that $\beta(x,y)\leq C$ for all $x,y\in ]0,R]$. Then, substituting (\ref{E:disc B_i}) into (\ref{Bcap}), we have
\begin{eqnarray*}
\|\hat{\mathbf{B}}({\mathbf{N}})-\hat{\mathbf{B}}({\hat{\mathbf{N}}})\| &\leq &\frac{1}{2}C\sum_{i=1}^{I}  \sum_{j=1}^{i-1} \sum_{x_{i-3/2}\leq x_j+x_k < x_{i-1/2}} |N_j N_k-\hat{N}_j \hat{N}_k|\nonumber\\
&&+\frac{1}{2}C\sum_{i=1}^{I}  \sum_{j=1}^{i} \sum_{x_{i-1/2}\leq x_j+x_k < x_{i+1/2}} |N_j N_k-\hat{N}_j \hat{N}_k|\nonumber\\
&& + \frac{1}{2}C\sum_{i=1}^{I}  \sum_{j=1}^{i+1} \sum_{x_{i+1/2}\leq x_j+x_k < x_{i+3/2}} |N_j N_k-\hat{N}_j \hat{N}_k|\nonumber\\
&\leq & \frac{3}{2}C\sum_{j=1}^{I}\sum_{k=1}^{I}|N_j N_k-\hat{N}_j\hat{N}_k|.
\end{eqnarray*}
Now we apply the following useful equality $N_j N_k-\hat{N}_j\hat{N}_k = \frac{1}{2}[(N_j+\hat{N}_j)(N_k-\hat{N}_k) + (N_j-\hat{N}_j)(N_k+\hat{N}_k)]$ to get
\begin{eqnarray}\label{1st lipschitz condition for B}
\hspace{.4in}\|\hat{\mathbf{B}}({\mathbf{N}})-\hat{\mathbf{B}}({\hat{\mathbf{N}}})\| \leq  \frac{3}{4}C\sum_{j=1}^{I}\sum_{k=1}^{I}\bigg[|(N_j+\hat{N}_j)| |(N_k-\hat{N}_k)| + |(N_j-\hat{N}_j)| |(N_k+\hat{N}_k)|\bigg].
\end{eqnarray}
It can be easily shown that the total number of particles decreases in a coagulation process, i.e.\
\begin{eqnarray*}
 \sum_{j=1}^{I}N_j \leq N_T^0:= \text{Total number of particles which are taken initially}.
\end{eqnarray*}
The  equation (\ref{1st lipschitz condition for B}) can be rewritten as
\begin{eqnarray}\label{lipschitz condition for B}
\|\hat{\mathbf{B}}({\mathbf{N}})-\hat{\mathbf{B}}({\hat{\mathbf{N}}})\| & \leq&  \frac{3}{2} N_T^0 C\bigg[ \sum_{k=1}^{I}|(N_k-\hat{N}_k)|+ \sum_{j=1}^{I}|(N_j-\hat{N}_j)|\bigg]\\
&& \nonumber\leq 3 N_T^0 C \|\mathbf{N}-\hat{\mathbf{N}}\|.
\end{eqnarray}
Similarly as before we can easily show the Lipschitz condition for death term as
\begin{eqnarray}\label{lipschitz condition for D}
 \|\hat{\mathbf{D}}({\mathbf{N}})-\hat{\mathbf{D}}({\hat{\mathbf{N}}})\|\leq 3 N_T^0 C \|\mathbf{N}-\hat{\mathbf{N}}\|.
\end{eqnarray}
Thus, Theorem \ref{convergence theorem} implies the convergence of the cell average technique and the convergence is of the same order as the consistency.

\section{Conclusions}\label{conc}
We have presented a detailed convergence analysis of the cell average technique for nonlinear continuous Smoluchowski coagulation equation. 
It is proved that the cell average technique is second order convergent on uniform grids. However, it gives only a first order convergence on non-uniform 
smooth geometric grids. To obtain a second order convergence, either one needs to adapt a different approach than the one presented here, or modify the error 
term $E_1$ which may lead to some improvements in CAT. It is also interesting to analyze CAT for nonlinear continuous SCE on more general grids, which we intend to study 
in future.

\end{document}